\documentclass[reqno]{amsart}

\usepackage{amssymb}

\usepackage{graphicx}
\usepackage[usenames,dvipsnames,svgnames]{xcolor}

\usepackage[colorlinks, urlcolor=Navy, linkcolor=Navy, citecolor=Navy]{hyperref}

\usepackage[alphabetic]{amsrefs}
\renewcommand{\MR}[1]{}

\usepackage[all]{xy}
\usepackage{tikz}
\usepackage{tikz-cd}
\usetikzlibrary{arrows,shapes,chains,matrix,positioning,scopes,cd,patterns}

\usepackage{todonotes}

\usepackage[truedimen,margin=25truemm]{geometry}

\numberwithin{equation}{section}

\newtheorem{thm}{Theorem}[section]
\newtheorem{cor}[thm]{Corollary}
\newtheorem{lem}[thm]{Lemma}
\newtheorem{prop}[thm]{Proposition}

\theoremstyle{definition}
\newtheorem{dfn}[thm]{Definition}

\newtheorem{example}[thm]{Example}

\newtheorem*{notation-convention}{Notation and convention}

\usepackage[inline, left, final]{showlabels} 

\newcommand{\mfm}{\mathfrak{m}}

\newcommand{\cA}{\mathcal{A}}
\newcommand{\cB}{\mathcal{B}}
\newcommand{\cC}{\mathcal{C}}

\newcommand{\cF}{\mathcal{F}}

\newcommand{\cJ}{\mathcal{J}}
\newcommand{\cK}{\mathcal{K}}

\newcommand{\cT}{\mathcal{T}}
\newcommand{\cU}{\mathcal{U}}
\newcommand{\cS}{\mathcal{S}}

\newcommand{\cX}{\mathcal{X}}
\newcommand{\cY}{\mathcal{Y}}

\newcommand{\sK}{\mathsf{K}}

\newcommand{\sT}{\mathsf{T}}

\newcommand{\Mod}{\mathsf{Mod}}

\renewcommand{\L}{\Lambda}

\newcommand{\add}{\operatorname{\mathsf{add}}\nolimits}
\newcommand{\ann}{\operatorname{\mathsf{ann}}\nolimits}

\newcommand{\Spec}{\mathsf{Spec}}
\newcommand{\End}{\operatorname{End}\nolimits}
\newcommand{\Ext}{\operatorname{Ext}\nolimits}
\newcommand{\Fac}{\operatorname{\mathsf{Fac}}\nolimits}
\newcommand{\Hom}{\operatorname{Hom}\nolimits}

\renewcommand{\Im}{\operatorname{Im}\nolimits}

\newcommand{\Ker}{\operatorname{Ker}\nolimits}
\newcommand{\Cok}{\operatorname{Cok}\nolimits}

\renewcommand{\mod}{\mathsf{mod}}
\newcommand{\rad}{\operatorname{rad}\nolimits}

\newcommand{\thick}{\operatorname{\mathsf{thick}}\nolimits}

\newcommand{\fp}{\mathsf{fp}}

\newcommand{\Fl}{\mathsf{fl}}
\newcommand{\Filt}{\mathsf{Filt}}

\newcommand{\proj}{\operatorname{\mathsf{proj}}\nolimits}

\newcommand{\Tr}{\operatorname{Tr}\nolimits}
\newcommand{\xto}[1]{\xrightarrow{#1}}

\newcommand{\ftors}{\mathsf{f}\text{-}\mathsf{tors}}
\newcommand{\twopsilt}{\mathsf{2}\text{-}\mathsf{psilt}}
\newcommand{\twosilt}{\mathsf{2}\text{-}\mathsf{silt}}
\newcommand{\silt}{\mathsf{m}\text{-}\mathsf{silt}}
\newcommand{\tors}{\mathsf{tors}}
\newcommand{\brick}{\mathsf{brick}}

\newcommand{\tautilt}{\mathsf{sin}\text{-}\mathsf{silt}}

\newcommand{\CM}{\mathsf{CM}}

\newcommand{\ideal}[1]{\left\langle #1 \right\rangle}
\begin{document}
\title[Tilting and silting theory of Noetherian algebras]{Tilting and silting theory of Noetherian algebras}
\author[Y. Kimura]{Yuta Kimura}
\address{Fakult\"{a}t f\"{u}r Mathematik, Universit\"{a}t Bielefeld, 33501 Bielefeld, Germany} 
\email{ykimura@math.uni-bielefeld.de}
\thanks{The author has been supported by the Alexander von Humboldt Foundation in the framework of an Alexander von Humboldt Professorship 
endowed by the German Federal Ministry of Education and Research.}
\date{\today}
\begin{abstract}
We develop silting theory of a noetherian algebra $\L$ over a commutative noetherian ring $R$.
We study mutation theory of $2$-term silting complexes of $\L$, and as a consequence, we see that mutation exists.
As in the case of finite dimensional algebras, functorially finite torsion classes of $\L$ bijectively correspond to silting $\L$-modules, if $R$ is complete local.
We show a reduction theorem of $2$-term silting complexes of $\L$, and by using this theorem, we study torsion classes of the module category of $\L$.
When $R$ has Krull dimension one, we describe the set of torsion classes of $\L$ explicitly by using the set of torsion classes of finite dimensional algebras.
\end{abstract}
\maketitle
\tableofcontents
\section{Introduction}
Tilting theory was originally introduced in the representation theory of  algebras, and it plays a central role in the representation theory.
One of the important result of a tilting module or a tilting complex over an algebra $A$ is that it induces derived equivalences between $A$ and its endomorphism algebra \cite{Happel, Rickard}.
Recently, there are many studies of derived categories, which appear in various areas of mathematics.
Among these studies, tilting theory plays an important role.
More basically, a tilting module induces an equivalence between a torsion class generated by the tilting module and a certain torsion free class over the endomorphism algebra \cite{Brenner-Butler}.
Many important results on tilting modules over algebras, including the above equivalences, have developed until now, see \cite{AHK} for instance.

One of a useful tool to study tilting modules is mutation, which was introduced by Riedtmann and Schofield \cite{RS}.
Mutation is an operation to construct a new tilting module from a given one by replacing an indecomposable direct summand.
The concept of mutation was introduced for cluster tilting objects in $2$-Calabi-Yau categories \cite{BMRRT, Iyama-Yoshino} and played an important role in categorifications of cluster algebras.

Silting complexes over an algebra are natural generalization of tilting complexes, which were introduced by Keller and Vossieck \cite{KV}.
It was shown by Aihara and Iyama \cite{Aihara-Iyama} that mutation works well on silting complexes.
In particular, mutation behaves quite nicely on $2$-term silting complexes.
For a finite dimensional algebra over a field, $2$-term silting complexes correspond bijectively to a special class of modules called support $\tau$-tilting modules \cite{Adachi-Iyama-Reiten}.


There are many studies of combinatorics of tilting modules and support $\tau$-tilting modules.
For a finite dimensional algebras, see for instance \cite{Adachi, DIRRT, Mizuno}.
One important class of algebras, which is not artinian and has a nice combinatorics of tilting modules, is preprojective algebras of non-Dynkin type.
The algebra admits families of tilting modules parameterized by Coxeter groups \cite{BIRSc, Kimura-Mizuno}.
These tilting modules play a central role in a categorification of cluster algebras from Lie theory \cite{GLS06, GLS13}.

Let $R$ be a commutative noetherian ring.
A \emph{module-finite $R$-algebra} (or \emph{noetherian algebra}) is an $R$-algebra $\L$ which is finitely generated as an $R$-module.
When $\L$ is a module-finite $R$-algebra which is $d$-Calabi-Yau, tilting modules over $\L$ were studied in \cite{Iyama-Reiten, Iyama-Wemyss}.
In this case, tilting modules have a closed connection with Cohen-Macaulay representations of $R$ and the theory of non-commutative crepant resolutions of Van den Bergh.

Although there are many important studies of noetherian algebras and their tilting modules, tilting theory of noetherian algebras seems to be not well prepared.
In this paper, based on support $\tau$-tilting theory, we develop tilting theory of noetherian algebras.
In addition, we apply tilting theory to study the classification problem of torsion classes of the category of finitely generated modules.

In \cite{Adachi-Iyama-Reiten} for a finite dimensional algebra $A$, the following results were proved.
\begin{itemize}
\item[(i)]
Mutation of support $\tau$-tilting $A$-modules exist.
\item[(ii)]
There is a bijection between support $\tau$-tilting $A$-modules and functorially finite torsion classes of $\mod A$.
\item[(iii)]
There is a bijection between support $\tau$-tilting $A$-modules and $2$-term silting complexes of $A$.
\end{itemize}
We see that the above statements hold for a module-finite $R$-algebra with a complete local noetherian ring $R$.
It is known that for a suitable setting, similar statements of (ii) and (iii) hold \cite{AngeleriHugel-Marks-Vitoria, Iyama-Jorgensen-Yang}.
Therefore we first study mutation property over noetherian algebras.
If $\L$ is a noetherian algebra, then the Auslander-Reiten translation $\tau$ no longer exists.
Therefore we need to modify the definition of support $\tau$-tilting modules.
We study silting modules defined as follows (for notation used below, see Section \ref{section-mutations}).

\begin{dfn}[Definition \ref{dfn-silting-module}]\label{intro-dfn-silting-module}
Let $\L$ be a ring.
We say that a $\L$-module $M$ is a \emph{silting module} if there exists a $2$-term silting complex $P$ in $\sK^{\rm b}(\proj \L)$ such that $\add M=\add H^0(P)$.
\end{dfn}

The name``silting module" was first used in \cite{AngeleriHugel-Marks-Vitoria} and its definition contains infinitely generated modules.
Our silting modules are finitely presented by definition.

In \cite{Adachi-Iyama-Reiten, Iyama-Jorgensen-Yang}, a bijection between $2$-term silting complexes and silting modules (or, support $\tau$-tilting modules) was given.
Mutation of silting modules (or silting pairs) are defined by using this bijection.
An advantage of considering silting modules is that they enable us to calculate mutation concretely, and to apply $\tau$-tilting theory of finite dimensional algebras for study of torsion classes of noetherian algebras.
On the other hand, an advantage of considering $2$-term silting complexes is that it is easier to prove theorems by dealing with complexes than dealing with modules for a while, and that several results are known in $2$-term silting complexes.

From now on, we explain our results in this paper.
For a $2$-term presilting complex $P$ in $\cK=\sK^{\rm b}(\proj\L)$, a \emph{completion} of $P$ is a $2$-term silting complex which has $P$ as a direct summand.
If there exists a right $(\add P)$-approximation $f$ of $\L[1]$ in $\cK$, then $P\oplus (C(f)[-1])$ is a completion of $P$, which we call the Bongartz completion of $P$.
The co-Bongartz completion of $P$ is defined dually.

In Section \ref{section-mutations}, we first construct mutation theory of $2$-term silting complexes over noetherian algebras.
The following result shows that (co-)Bongartz completions exist for noetherian algebras, though it fails for arbitrary rings, see Example \ref{example-not-noetherian}.
\begin{thm}\label{intro-thm-noeth-mutation}
Let $\L$ be a module-finite $R$-algebra and $P$ a $2$-term presilting complex in $\sK^{\rm b}(\proj\L)$.
\begin{itemize}
	\item[{\rm (a)}] There exist the Bongartz completion and the co-Bongartz completion of $P$.
	\item[{\rm (b)}] Assume that $\sK^{\rm b}(\proj\L)$ is Krull-Schmidt and $P$ is almost complete. Then the co-Bongartz completion of $P$ is an irreducible left mutation of the Bongartz completion of $P$.
\end{itemize}
\end{thm}
If $\L$ is a noetherian algebra over a complete local ring, then $\sK^{\rm b}(\proj\L)$ is Krull-Schmidt.
Therefore Theorem \ref{intro-thm-noeth-mutation} (b) says that mutation of $2$-term silting complexes exist in this setting.
Our second result shows that it holds without assuming that $\L$ is a noetherian algebra.
\begin{thm}[Lemma \ref{lem-at-most-two} and Proposition \ref{prop-2-term-at-most-two}]\label{intro-thm-KS-mutation}
Let $\L$ be a ring such that $\sK^{\rm 2}(\proj\L)$ is Krull-Schmidt and $P$ an almost complete $2$-term silting complex.
\begin{itemize}
\item[(a)]
If there exists a completion of $P$, then it is either the Bongartz completion of $P$ or the co-Bongartz completion of $P$.
\item[(b)]
If $P$ admits both the Bongartz completion $S$ and the co-Bongartz completion $T$, then $T$ is an irreducible left mutation of $S$.
\end{itemize}
\end{thm}
In Subsection \ref{subsection-basic-supp-tau-tilt}, we show how to calculate mutation of silting modules (or pairs) explicitly, see Theorem \ref{thm-sttilt-mutation}.
In Subsection \ref{subsection-mutation-hasse-quiver}, we see that the exchange quiver of $2$-term silting complexes corresponds to the Hasse quiver of that under the assumption (M) on a ring $\L$, see Theorem \ref{thm-Hasse-mutation-quiver}.

In Sections \ref{section-support-tau-tilting-and-torsion-class} and \ref{section-reduction}, we concentrate on a noetherian algebra $\L$ over a commutative complete local noetherian ring.
A bijection between functorially finite torsion classes of $\mod\L$ and silting $\L$-modules is shown in Section \ref{section-support-tau-tilting-and-torsion-class}, which is originally shown by Adachi, Iyama and Reiten \cite{Adachi-Iyama-Reiten} when $\L$ is a finite dimensional algebra.
\begin{thm}[Theorem \ref{thm-bij-sttilt-ftors}]\label{thm-bij-sttilt-ftors-intro}
Let $R$ be a complete local noetherian ring and $\L$ a module-finite $R$-algebra.
A map $M\mapsto \Fac M$ gives a bijection from the first of the following set to the second:
\begin{itemize}
\item
The set $\silt\L$ of isomorphism classes of basic silting $\L$-modules.
\item
The set $\ftors\L$ of functorially finite torsion classes of $\mod\L$.
\end{itemize}
\end{thm}
In Section \ref{section-reduction}, we study a reduction theorem.
\begin{thm}[Theorem \ref{thm-reduction}]\label{thm-intro-reduction}
Let $(R, \mfm)$ be a commutative complete local noetherian ring, $\L$ a module-finite $R$-algebra and $I$ an ideal of $\L$ such that $I\subseteq \mfm\L$. 
Then a map $M\mapsto M/IM$ gives a bijection from the first of the following set to the second:
\begin{itemize}
\item
The set $\silt\L$ of isomorphism classes of basic silting $\L$-modules.
\item
The set $\silt(\L/I)$ of isomorphism classes of basic silting $\L/I$-modules.
\end{itemize}
Moreover, this bijection preserves mutation.
\end{thm}
Note that Theorem \ref{thm-intro-reduction} is a generalization of a result due to \cite{EJR}, where it was assumed that $R$ is artin and $I$ is nilpotent.
We refer \cite{Gnedin} in the case where $I$ is generated by a regular sequence of $R$ and silting complexes of arbitrary length, see also \cite{Eisele}.

In Section \ref{section-torsion-fl}, we study the set $\tors \L$ of all torsion classes of $\mod\L$.
In the case where $\L=R$, the classification problem of subcategories of $\mod R$ was developed by many researchers \cite{Gabriel, Stanley-Wang}.
In this case, torsion classes of $\mod R$ bijectively correspond to specialization closed subsets of $\Spec R$.
On the other hand, the case where $\L$ is non-commutative, the set $\tors\L$ has a rich structure in general, for example \cite{Ingalls-Thomas}.
Therefore $\tors\L$ can not be controlled by $\Spec R$.

We first investigate the set $\tors(\Fl\L)$ of all torsion classes of the category $\Fl\L$ of finite length $\L$-modules.
\begin{thm}[Theorem \ref{thm-torsion-fl}]\label{thm-intro-torsion-fl}
Let $(R, \mfm)$ be a commutative local noetherian ring and $\L$ a module-finite $R$-algebra.
Then there exists a bijection
\[
\tors(\Fl\L) \longrightarrow \tors(\L/\mfm\L)
\]
given by $\cT \mapsto \cT\cap \mod(\L/\mfm\L)$.
\end{thm}
This theorem gives a useful bridge between studies of torsion classes of module-finite algebras and finite dimensional algebras.

Next we observe all torsion classes of $\mod\L$.
In the study of $\tors\L$, an important tool is the localization functor at a prime ideal of $R$.
So if $R$ is a special class of commutative noetherian rings, then $\tors \L$ becomes easier as follows.
We assume that $R$ is a commutative local noetherian integral domain with Krull dimension one.
Let $K$ be the fractional field of $R$.
Then for a module-finite $R$-algebra $\L$, we have a finite dimensional $K$-algebra $A:=K\otimes_{R}\L$.
For a subcategory $\cC$ of $\mod\L$, let $\overline{\cC}=\{M\in\mod\L \mid \Fac M \cap\Fl\L\subseteq\cC\}$ be a subcategory of $\mod\L$, and let $K\cC:=\{ K\otimes_{R}M \mid M\in\cC \}$ be a subcategory of $\mod A$.
For a torsion class $\cT$ of $\mod\L$, let $\bigl[\cT, \, \overline{\cT}\bigr]_{\L}=\{\cC\in\tors\L \mid \cT \subseteq \cC \subseteq \overline{\cT}\}$.
\begin{thm}[Theorem \ref{thm-torsion-disjoint-union}]\label{torsion-disjoint-union-intro}
Let $R$ be a commutative local integral domain with Krull dimension one, and $\L$ a module-finite $R$-algebra.
Let $A=K\otimes_R \L$.
Then we have a disjoint union
\[
\tors\L = \coprod_{\cT\in\tors(\Fl\L)}\,\bigl[\cT, \, \overline{\cT}\bigr]_{\L},
\]
and there is a bijection
\[
\coprod_{\cT\in\tors(\Fl\L)}\,\bigl[\cT, \, \overline{\cT}\bigr]_{\L} \xto{K(-)} \coprod_{\cT\in\tors(\Fl\L)}\,\bigl[0, \, K(\overline{\cT})\bigr]_{A}.
\]
\end{thm}
If $\L/\mfm\L$ is a $\tau$-tilting finite algebra \cite{DIJ}, then the description becomes easier, see Corollary \ref{cor-finiteness-criterion}.
In \cite{Iyama-Kimura}, we observe the map of Theorem \ref{torsion-disjoint-union-intro} for an arbitrary commutative noetherian ring $R$.

As a corollary of Theorem \ref{torsion-disjoint-union-intro}, we have the following result.
\begin{cor}[Corollary \ref{cor-finiteness}]
Assume moreover that $(R, \mfm)$ is a complete local noetherian integral domain with Krull dimension one.
Then $\tors\L$ is a finite set if and only if $\tors(\L/\mfm\L)$ and $\tors A$ are finite sets.
\end{cor}
In Section \ref{section-example}, we calculate some examples by using the above theorems.
Let $k[[x]]$ be the ring of formal power series in one variable with a field $k$.
We denote by $(x)$ the maximal ideal of $k[[x]]$.
\begin{thm}[Theorem \ref{thm-tors-hereditary}]\label{intro-thm-hered}
Let $R=k[[x]]$, $\mfm=(x)$.
Let $\L$ be an $R$-algebra defined by the following $(n\times n)$-matrix form $(n\geq 1)$:
\begin{eqnarray*}
	\L=\left[\begin{array}{ccccc}
		R & R & & \cdots & R \\
		\mfm & R & & & \vdots \\
		& \mfm & \ddots & &  \\
		\vdots & & \ddots  & R & R \\
		\mfm & \cdots & & \mfm & R
	\end{array}
	\right] \subset \mathrm{M}_n(R).
\end{eqnarray*}
Then we have
\[
|\silt\L| = \begin{pmatrix}
2n\\
n
\end{pmatrix},
\qquad
| \tors \L | =\frac{3}{2} \begin{pmatrix}
2n\\
n
\end{pmatrix}.
\]
\end{thm}
\begin{thm}[Theorem \ref{thm-msilt-Bass-V-Aus}]\label{intro-thm-msilt-Bass-V-Aus}\
Let $R=k[[x]]$, $\mfm=(x)$ and $n$ a non-negative integer.
Let $\L = \begin{bmatrix} R & R \\ \mfm^{n} & R \end{bmatrix}$.
We denote by $\Gamma$ the Auslander algebra of the category $\CM\L$ of $\L$-lattices (see Subsection \ref{subsection-Bass-V} for details).
Then we have an isomorphism of posets
\[
\silt\Gamma \simeq \mathfrak{S}_{n+2},
\]
where $\mathfrak{S}_{n+2}$ is the symmetric group of degree $n+2$ which is a poset by the right weak order.
\end{thm}
\begin{notation-convention}
All subcategories are assumed to be full and closed under isomorphisms.

For an object $M$ of an additive category $\cC$, we denote by $\add M$ the full subcategory of $\cC$ consisting of direct summands of finite direct sums of copies of $M$.

An additive category is said to be \emph{Krull-Schmidt} if each object is a finite direct sum of objects such that their endomorphism algebras are local.
Let $M$ be an object of a Krull-Schmidt category and $M=\bigoplus_{i=1}^{l}M_{i}^{\oplus m_{i}}$ be a direct sum decomposition of $M$, where each $M_{i}$ is indecomposable and $M_{i}\not\simeq M_{j}$ if $i\neq j$.
Put $|M|:=l$ and we say that $M$ is \emph{basic} if $m_{i}=1$ for each $i=1,\ldots,l$.

We say that two objects $M, N$ in an additive category are \emph{additive equivalent} if $\add M = \add N$ holds.
If the additive category is Krull-Schmidt and $M$ and $N$ are basic, then they are additively equivalent if and only if they are isomorphic to each other.

Let $\cC$ be a Krull-Schmidt category.
A morphism $f : X \to Y$ in $\cC$ is said to be \emph{right minimal} if $f$ does not have a direct summand of the form $Z \to 0$.
It is known that $f$ is right minimal if and only if each morphism $g : X \to X$ which satisfies $fg=f$ is an isomorphism.
Dually, we define \emph{left minimal} morphisms.

Let $\cC$ be a subcategory of an additive category $\cB$ and $B\in\cB$.
A \emph{right $\cC$-approximation} of $B$ is a morphism $f: C \to B$ such that $C\in\cC$ and the map $f^{\ast}|_{\cC} : \Hom_{\cB}(-,C)|_{\cC} \to \Hom_{\cB}(-,B)|_{\cC}$ is surjective.
We say that $\cC$ is \emph{contravariantly finite} in $\cB$ if each object in $\cB$ has a right $\cC$-approximation.
Dually, we define a \emph{left $\cC$-approximation} of $B$ and \emph{covariantly finite} subcategories in $\cB$. 
A subcategory of $\cB$ is said to be \emph{functorially finite} if both covariantly and contravariantly finite in $\cB$.

Let $\cC$ be an additive category.
An \emph{ideal} $I$ of $\cC$ is a class of morphisms of $\cC$ such that $I(X, Y)$ is an additive subgroup of $\Hom_{\cC}(X, Y)$ for any $X, Y\in\cC$ and satisfies $hgf\in I(W, Z)$ for any $f\in\Hom_{\cC}(W, X)$, $g\in I(X, Y)$ and $h \in \Hom_{\cC}(Y, Z)$.
We denote by $\cJ_{\cC}$ the \emph{Jacobson radical} of $\cC$, see \cite[Appendix]{ASS} for instance.

Let $A$ be a ring.
We denote by $\Mod A$ (resp. $\mod A$, $\fp A$, $\proj A$) the category of (resp. finitely generated, finitely presented, finitely generated projective) left $A$-modules and denote by $\sK^{\rm b}(\proj A)$ the bounded homotopy category of $\proj A$.
For $M\in\mod A$, we denote by $\Fac M$ the full subcategory of $\mod A$ consisting of factor modules of modules in $\add M$.

Let $R$ be a commutative ring.
An $R$-algebra is a ring $\L$ with a ring homomorphism $\phi : R \to \L$ such that the image of $\phi$ is contained in the center of $\L$.
For a subset $\cS\subset\L$, we denote by $\langle \cS \rangle$ the two-sided ideal of $\L$ generated by $\cS$.

For a triangulated category $\cT$ and a morphism $f$ in $\cT$, we denote by $C(f)$ the mapping cone of $f$ in $\cT$.

For a posets $X, Y$, a morphisms of posets is a map $f : X \to Y$ such that $a\leq b$ in $X$ always implies $f(a)\leq f(b)$.
A morphism $f : X \to Y$ of posets is called an \emph{embedding of posets} if $f(a) \leq f(b)$ implies $a\leq b$ for any $a, b\in X$.
\end{notation-convention}
\section{Silting complexes and mutation}\label{section-mutations}
In this section, we study mutation of $2$-term silting complexes by observing the (co-)Bongartz completions of almost complete $2$-term silting complexes.
\subsection{Silting complexes}\label{subsection-silting}
Let $\L$ be a ring.
We first recall the definition of silting complexes in $\cK=\sK^{\rm b}(\proj\L)$, and recall some properties of them.
For an object $P\in\cK$, we denote by $\thick_{\cK}P$ the smallest triangulated subcategory of $\cK$ containing $P$ and closed under direct summands.
A complex $P=(P_{i}, d_{i}: P_{i} \to P_{i+1})\in\cK$ is said to be \emph{2-term} if $P_{i}=0$ for $i\neq 0,-1$.
We denote by $\cK^2=\sK^2(\proj\L)$ the subcategory of $\cK$ consisting of $2$-term complexes.
\begin{dfn}\label{def-silting}
Let $P\in\cK=\sK^{\rm b}(\proj\L)$.
\begin{itemize}
\item[(1)]
$P$ is \emph{presilting} in $\cK$ if $\Hom_{\cK}(P,P[i])=0$ for $i>0$.
\item[(2)]
$P$ is \emph{silting} in $\cK$ if $P$ is presilting and $\thick_{\cK}P=\cK$ holds.
\item[(3)]
A \emph{$2$-term silting} (resp. \emph{$2$-term presilting}) complex in $\cK$ is a silting (resp. presilting) complex in $\cK$ which is a $2$-term complex.
\item[(4)]
Assume that $P$ is a $2$-term presilting complex.
A \emph{completion} of $P$ is a $2$-term silting complex $T$ of $\L$ satisfying $\add P \subset \add T$.
\end{itemize}
We denote by $\mathsf{silt}\L$ (resp. $\twosilt\L$, $\twopsilt\L$) the set of additive equivalence classes of silting (resp. $2$-term silting, $2$-term presilting) complexes in $\cK$.
\end{dfn}
The following lemma gives a sufficient condition such that a given $P\in\twopsilt\L$ admits a completion.
\begin{lem}\label{lem-Bongartz-completion-silting}
Let $P\in\twopsilt\L$.
The following statements hold.
\begin{itemize}
\item[{\rm (a)}]
Assume that there exists a right $(\add P)$-approximation $f$ of $\L[1]$ in $\cK$.
Then $P\oplus C(f)[-1]$ is a $2$-term silting complex in $\cK$.
\item[{\rm (b)}]
Assume that there exists a left $(\add P)$-approximation $g$ of $\L$ in $\cK$.
Then $P\oplus C(f)$ is a $2$-term silting complex in $\cK$.
\end{itemize}
\end{lem}
\begin{proof}
See \cite[Proposition 2.16]{Aihara} and \cite[Lemma 4.2]{Iyama-Jorgensen-Yang}.
\end{proof}
For a $P\in\twopsilt\L$, if there exists a right $(\add P)$-approximation $f$ of $\L[1]$ in $\cK$, then we call $P\oplus C(f)[-1]$ the \emph{Bongartz completion} of $P$.
Dually, if there exists a left $(\add P)$-approximation $g$ of $\L$ in $\cK$, then we call $P\oplus C(g)$ the \emph{co-Bongartz completion} of $P$.
It is fundamental that there exists a right $(\add P)$-approximation of $\L[1]$ in $\cK$ if and only if $\Hom_{\cK}(P,\L[1])$ is a finitely generated right $\End_{\cK}(P)$-module.
\begin{example}
Let $R$ be a commutative noetherian ring and $\L$ a module-finite $R$-algebra.
Then for any complexes $P\in\cK$, $\add P$ is functorially finite in $\cK$.
Thus any $2$-term presilting complex of $\L$ admits both the Bongartz completion and the co-Bongartz completion.
\end{example}
It is known that $\mathsf{silt}\L$ is a poset (and so is $\twosilt\L$) by the following binary relation.

For $X,Y\in\cK$, we write $X\geq Y$ if $\Hom_{\cK}(X,Y[i])=0$ for $i>0$.
Silting objects and this relation $\geq$ were studied by Aihara and Iyama.
\begin{prop}\label{prop-silting-indec-number}\cite{Aihara-Iyama}
The relation $\geq$ gives a partial order on $\mathsf{silt}\L$.
\end{prop}
The Bongartz completion (resp. co-Bongartz completion) of $P$ satisfies the following maximal (resp. minimal) condition.
In particular, it is independent of a choice of a right $(\add P)$-approximation (resp. a left $(\add P)$-approximation) up to additive equivalence.
\begin{lem}\label{lem-Bong-max}
Let $P\in\twopsilt\L$.
The following statements hold.
\begin{itemize}
\item[(a)]
Assume that there exists the Bongartz completion $P\oplus Q$ of $P$.
If $P\oplus U$ is a $2$-term silting complex, then we have $P\oplus Q \geq P\oplus U$.
\item[(b)]
Assume that there exists the co-Bongartz completion $P\oplus Q$ of $P$.
If $P\oplus U$ is a $2$-term silting complex, then we have $P\oplus Q \leq P\oplus U$.
\end{itemize}
\end{lem}
\begin{proof}
This follows from a direct calculation.
\if()
Since $P\oplus Q$ and $P\oplus U$ are silting complexes, we have $\Hom_{\cK}(P,P[i])=0$, $\Hom_{\cK}(P,U[i])=0$ and $\Hom_{\cK}(Q,P[i])=0$ for $i>0$.
Since $P\oplus Q$ is the Bongartz completion of $P$, there exists a triangle $\L \to Q \to P^{\prime} \to \L[1]$, where $P^{\prime} \to \L[1]$ is a right $(\add P)$-approximation of $\L[1]$.
By applying the functor $\Hom_{\cK}(-,U[i])$ to this triangle, we have $\Hom_{\cK}(Q,U[i])=0$ for $i>0$.
Therefore we have $P\oplus Q \geq P\oplus U$.
\fi
\end{proof}
Let $I\subset \L$ be a two-sided ideal.
For a $2$-term complex $P=(P_{-1}\xto{d} P_{0})$ in $\cK(\L)=\sK^{\rm b}(\proj\L)$, let
\[\overline{P}=(\L/I)\otimes_{\L}P=(P_{-1}/IP_{-1}\to P_{0}/IP_{0})\]
be a $2$-term complex in $\cK(\L/I)=\sK^{\rm b}(\proj\L/I)$.
\begin{lem}\label{lem-two-term-fac}
Let $I\subset \L$ be a two-sided ideal.
The following statements hold.
\begin{itemize}
\item[(a)]
For $2$-term complexes $P, Q\in\cK(\L)$, $\Hom_{\cK(\L)}(P,Q[1])=0$ implies $\Hom_{\cK(\L/I)}(\overline{P},\overline{Q}[1])=0$.
\item[(b)]
We have a map from $\twopsilt\L$ to $\twopsilt(\L/I)$ given by $P\mapsto \overline{P}$, which induces a morphism of posets
\[
\twosilt\L \longrightarrow \twosilt(\L/I).
\]
\end{itemize}
\end{lem}
\begin{proof}
(a)
Let $\overline{\L}:=\L/I$, $P=(P_{-1}\xto{\alpha} P_{0})$ and $Q=(Q_{-1}\xto{\beta} Q_{0})$.
We have a commutative diagram with exact rows
\[
\begin{tikzcd}
\Hom_{\L}(P_{-1}, Q_{-1}) \times \Hom_{\L}(P_{0},Q_{0}) \ar[r, "\gamma_{\alpha, \beta}"] \ar[d, "\chi"] & \Hom_{\L}(P_{-1},Q_{0}) \ar[r]\ar[d, "\psi"] & \Hom_{\cK(\L)}(P,Q[1]) \ar[r] \ar[d] & 0 \\
\Hom_{\overline{\L}}(\overline{P}_{-1}, \overline{Q}_{-1}) \times \Hom_{\overline{\L}}(\overline{P}_{0},\overline{Q}_{0})  \ar[r, "\gamma_{\overline{\alpha}, \overline{\beta}}"] & \Hom_{\overline{\L}}(\overline{P}_{-1},\overline{Q}_{0}) \ar[r] & \Hom_{\cK(\overline{\L})}(\overline{P},\overline{Q}[1]) \ar[r] & 0,
\end{tikzcd}
\]
where $\gamma_{\alpha, \beta}(f,g):=g\circ \alpha-\beta\circ f$ for $(f,g)\in\Hom_{\L}(P_{-1}, Q_{-1}) \times \Hom_{\L}(P_{0},Q_{0})$.
It is easy to see that $\psi$ is surjective.
Since $\psi$ is surjective, the right vertical morphism is surjective.
Therefore (a) holds.

The assertion (b) directly follows from (a).
\end{proof}
\subsection{Mutation of $2$-term silting complexes}\label{subsection-alm-comp}
Let $\cK^2=\sK^{2}(\proj\L)$ be the category of $2$-term complexes of $\cK=\sK^{\rm b}(\proj\L)$.
If $\cK$ is a Krull-Schmidt category, then so is $\cK^2$, since $\cK^2$ is closed under direct summands in $\cK$.
In this subsection, we assume that $\cK^2$ is a Krull-Schmidt category.
For example, this is the case if $\L$ is one of the following algebras.
\begin{example}\label{exampe-krull-schmidt}
Let $K$ be a field.
If $\L$ is one of the following rings, then $\sK^{\rm b}(\proj \L)$ is a Krull-Schmidt category.
\begin{itemize}
\item[(a)]
A finite dimensional $K$-algebra $\L$.

\item[(b)]
A module-finite $R$-algebra $\L$, where $R$ is a commutative complete local noetherian ring.

\item[(c)]
A complete Jacobian $K$-algebra $\L$ of a quiver with potentials, where the quiver is finite.
\end{itemize}
In fact, these algebras satisfy the assumption (F) of \cite[Section 4]{Kimura-Mizuno}, thus we can see that the homotopy categories of such algebras are Krull-Schmidt by \cite[Corollary 4.6]{Kimura-Mizuno}.
\end{example}
When $\cK$ is a Krull-Schmidt category, the following properties are known.
\begin{prop}\label{prop-presilt-silt-number}\cite{Aihara-Iyama}
Assume that $\cK=\sK^{\rm b}(\proj\L)$ is a Krull-Schmidt category.
The following statements hold.
\begin{itemize}
\item[(a)]
For any $P\in\mathsf{silt}\L$, indecomposable direct summands of $P$ give a free basis of the Grothendieck group of $\cK$.
In particular, $|P|=|\L|$ holds.
\item[(b)]
Let $P$ be a presilting complex such that there exists a silting complex having $P$ as a direct summand.
Then $P$ is a silting complex if and only if $|P|=|\L|$ holds.
\end{itemize}
\end{prop}
We say that $P\in\twopsilt\L$ is an \emph{almost complete silting complex} if $|P|=|\L|-1$ holds.

We see that an almost complete $2$-term silting complex admits a completion if and only if it admits either the Bongartz completion or the co-Bongartz completion.
The proofs of Theorems \ref{intro-thm-noeth-mutation} and \ref{intro-thm-KS-mutation} are given in this subsection.

Since $\cK^2$ is Krull-Schmidt, $\twosilt\L$ bijectively corresponds to the set of isomorphism classes of basic $2$-term silting complexes.
Thus we identify them.

For subcategories $\cX$ and $\cY$ of $\cK$, we denote by $\cX\ast\cY$ the subcategory of $\cK$ consisting of objects $Z$ such that there exists a triangle $X \to Z \to Y \to X[1]$ with $X\in\cX$ and $Y\in\cY$.
We can show the following lemma by applying some propositions in \cite{Aihara-Iyama}.
Here for the convenience of readers, we write the proof.
\begin{lem}\label{lem-exist-triangle}
Let $S\in\twopsilt\L$ and $T\in\twosilt\L$.
Assume that $S\geq T$.
Then there exists a triangle 
\begin{align}\label{triangle-silting-approximation}
S \xto{f} T_{0} \to T_{1} \to S[1]
\end{align}
which satisfies that $f$ is a left minimal $(\add T)$-approximation of $S$ and $\add T_{0}\cap\add T_{1}=0$.
If moreover $S$ is silting, then $\add(T_{0}\oplus T_{1})=\add T$ holds.
\end{lem}
\begin{proof}
Since $T$ is silting, by \cite[Proposition 2.23]{Aihara-Iyama}, $\cK=\bigcup_{i\geq 0}\add T[-i]\ast\add T[-i+1]\ast\cdots\ast\add T[i]$ holds.
We have $\L\geq S \geq T \geq \L[1]$.
Therefore $S\in\add T[-1] \ast \add T$ holds.
There exists a triangle
\[
T_{1}^{\prime}[-1] \to S \xto{f^{\prime}} T_{0}^{\prime} \xto{g^{\prime}}  T_{1}^{\prime},
\]
where $T_{0}^{\prime}, T_{1}^{\prime}\in\add T$.
Since $T$ is silting, $f^{\prime}$ is a left $(\add T)$-approximation of $S$.
We can write $f^{\prime}={^{t}}(f, 0) : S \to T_{0}\oplus T_{0}^{\prime\prime}$, where $f$ is a  left minimal $(\add T)$-approximation of $S$.
Then we have a triangle
\[
T_{1}[-1] \to S \xto{f} T_{0} \xto{g}  T_{1},
\]
such that $g$ is in $\cJ_{\cK}$ and $T_{1}$ is a direct summand of $T_{1}^{\prime}$.
To show $\add T_{0}\cap\add T_{1}=0$, it is enough to see that any morphism $a : T_{0} \to T_{1}$ belongs to $\cJ_{\cK}$.
Since $S$ is presilting, we have the following commutative diagram.
\[
\begin{tikzcd}
S \ar[r, "f"] \ar[d, "b"] & T_{0} \ar[r, "g"] \ar[d, "a"] & T_{1} \ar[r] \ar[d, "c"] & S[1] \ar[d] \\
T_{0} \ar[r, "g"] & T_{1} \ar[r] & S[1] \ar[r] & T_{0}[1]
\end{tikzcd}
\]
Since $f$ is a left $(\add T)$-approximation, there exists $d : T_{0} \to T_{0}$ such that $df=b$.
Since $(a-gd)f=0$, there exists $h : T_{1} \to T_{1}$ such that $a=gd+hg$.
Therefore $a$ belongs to $\cJ_{\cK}$.

If $S$ is silting, then by Proposition \ref{prop-silting-indec-number} (b), $\add(T_{0}\oplus T_{1})=\add T$ holds.
\end{proof}
Lemma \ref{lem-exist-triangle} induces a characterization of an almost complete $2$-term silting complex which is a direct summand of a $2$-term silting complex as follows.
\begin{lem}\label{lem-at-most-two}
Let $P$ be an almost complete $2$-term silting complex in $\cK$.
The following statements hold.
\begin{itemize}
\item[(a)]
If there exists a completion of $P$, then it is either the Bongartz completion or the co-Bongartz completion.
\item[(b)]
There exist at most two isomorphism classes of basic $2$-term silting complexes which are completions of $P$.
\item[(c)]
If $P$ admits both the Bongartz completion $S$ and the co-Bongartz completion $T$, then $S$ and $T$ are not additive equivalent.
\end{itemize}
\end{lem}
\begin{proof}
(a)
Assume that there exists a $2$-term silting complex $T=P\oplus U$.
We may assume that $U$ is indecomposable.
Since $T$ is $2$-term, $\L\geq T$ holds.
By Lemma \ref{lem-exist-triangle}, there exists a triangle $\L \to T_0 \to T_1 \to \L[1]$ satisfying $\add T_0 \cap \add T_1 =0$ and $\add(T_0\oplus T_1)=\add T$.
Either one of $U\in \add T_0$ or $U\in \add T_1$ holds.
The statement $U\in \add T_0$ implies the existence of a right $(\add P)$-approximation of $\L[1]$. In this case, $T$ is the Bongartz completion of $P$.
The statement $U\in \add T_1$ implies the existence of a left $(\add P)$-approximation of $\L$.
In this case, $T$ is the co-Bongartz completion of $P$.

(b)
By Lemma \ref{lem-Bong-max}, the Bongartz completion (resp. the co-Bongartz completion) of $P$ is  unique up to additive equivalence.
Thus the assertion directly follows from (a).

(c)
Let $g : P^{\prime} \to \L[1]$ be a right minimal $(\add P)$-approximation of $\L[1]$ and $f^{\prime} : \L \to P^{\prime\prime}$ be a left minimal $(\add P)$-approximation of $\L$.
We have triangles
\begin{align*}
\L \to X \to P^{\prime} \xto{g}\L[1], \qquad  \L \xto{f^{\prime}} P^{\prime\prime} \to Y \xto{g^{\prime}} \L[1],
\end{align*}
where $X=C(g)[-1]$ and $Y=C(f^{\prime})$.
Then $P\oplus X$ is the Bongartz completion of $P$, and $P \oplus Y$ is the co-Bongartz completion of $P$.
It is easy to see that $g$ is a right minimal $(\add P\oplus X)$-approximation of $\L[1]$ and $g^{\prime}$ is a right minimal $(\add P\oplus Y)$-approximation of $\L[1]$.
If $\add (P\oplus X) = \add (P\oplus Y)$ holds, then $Y \simeq P^{\prime}$.
This means that $|P\oplus Y|=|P|$, which is a contradiction.
Thus we have $\add (P\oplus X) \neq \add (P\oplus Y)$.
\end{proof}
We recall left and right mutations of silting complexes.
\begin{dfn}\label{dfn-mutation-silting}\cite[Definition 2.30, Theorem 2.31]{Aihara-Iyama}
Let $X\oplus Y\in\mathsf{silt}\L$ such that $\add X \cap \add Y = 0$.
\begin{itemize}
\item[(1)]
Assume that there exists a left $(\add X)$-approximation $f$ of $Y$ in $\cK$.
Then $X\oplus C(f)$ is a silting complex in $\cK$.
We call $X\oplus C(f)$ a \emph{left mutation of $X\oplus Y$ at $Y$}.
If $Y$ is indecomposable, then this mutation is said to be \emph{irreducible}.

\item[(2)]
Assume that there exists a right $(\add X)$-approximation $g$ of $Y$ in $\cK$.
Then $X\oplus C(g)[-1]$ is a silting complex in $\cK$.
We call $X\oplus C(g)[-1]$ a \emph{right mutation of $X\oplus Y$ at $Y$}.
If $Y$ is indecomposable, then this mutation is said to be \emph{irreducible}.

\item[(3)]
Two silting complexes are said to be \emph{mutations} of each other if these silting complexes have a common almost complete silting complex as a direct summand.
\end{itemize}
\end{dfn}
If an almost complete $2$-term silting complex $P$ admits both the Bongartz completion and the co-Bongartz completion, then we see that the co-Bongartz completion is an irreducible left mutation of the Bongartz completion.
\begin{prop}\label{prop-2-term-at-most-two}
Let $P$ be a basic almost complete $2$-term silting complex in $\cK$.
Assume that $P$ admits both the Bongartz  completion $P\oplus Q$ and the co-Bongartz completion $T=P\oplus U$.
We may assume that $Q$ and $U$ are indecomposable.
By Lemma \ref{lem-Bong-max}, $P\oplus Q\geq T$ holds.
Take a triangle
\begin{align}\label{triangle-QTT}
Q \xto{f} T_{0} \to T_{1} \to Q[1],
\end{align}
as (\ref{triangle-silting-approximation}) in Lemma \ref{lem-exist-triangle}, that is, $f$ is a left minimal $(\add T)$-approximation of $Q$, $\add T_{0} \cap \add T_{1}=0$ and $\add (T_0\oplus T_1) = \add T$.
Then the following statements hold.
\begin{itemize}
\item[(a)]
$f$ is a left minimal $(\add P)$-approximation of $Q$.
\item[(b)]
$U \simeq T_1$ holds.
\end{itemize}
\end{prop}
\begin{proof}
By adding a triangle $P \xto{\rm id} P \to 0 \to P[1]$ to the triangle (\ref{triangle-QTT}), we have a triangle
\[
P\oplus Q \xto{\left[ \begin{smallmatrix} {\rm id} & 0 \\ 0 & f \end{smallmatrix} \right]} P\oplus T_{0} \to T_{1} \to (P\oplus Q)[1],
\]
where $\left[ \begin{smallmatrix} {\rm id} & 0 \\ 0 & f \end{smallmatrix} \right]$ is a left minimal $(\add T)$-approximation of $P\oplus Q$.
By Lemma \ref{lem-exist-triangle}, $\add (P\oplus T_{0}) \cap \add T_{1}=0$ and $\add (P\oplus T_{0} \oplus T_{1})=\add T$ hold.
Since $U\notin\add P$, exactly one of $U\in\add T_0$ or $U\in\add T_1$ holds.
If $U\in\add T_0$, then $|P\oplus T_0|=|T|=|\L|$ holds.
This implies $\add T_1=0$.
Thus $f$ is an isomorphism, which is a contradiction.
We have $U\in \add T_1$.
Then $\add T_0\subset\add P$ holds, and $f$ is a left minimal $(\add P)$-approximation of $Q$.

Since $U\in \add T_1$ and $\add(P\oplus T_0)\cap\add T_1=0$, we have $\add U = \add T_1$.
Because $U$ is indecomposable, $U^{\oplus \ell}\simeq T_1$ for some $\ell \geq 1$.
The triangle (\ref{triangle-QTT}) implies that $[Q]+\ell[U]=[T_0]$ in the Grothendieck group of $\cK$.
By the dual statement of Lemma \ref{lem-exist-triangle}, there exists a triangle $U[-1] \to S_{1} \to S_{0} \xto{g} U$ such that $g$ is a right minimal $(\add S)$-approximation of $U$, $\add S_0 \cap \add S_1=0$ and $\add (S_0\oplus S_1)=\add S$.
By the similar argument as above, this triangle is equal to $U[-1] \to Q^{\oplus m} \to P^{\prime} \xto{g} U$ for some $m\geq 1$.
This triangle induces $m[Q]+[U]=[P^{\prime}]$ in the Grothendieck group of $\cK$.
By Proposition \ref{prop-presilt-silt-number} (a), $\ell=m=1$ holds.
Namely, $T_1\simeq U$ is indecomposable.
\end{proof}
\begin{proof}[Proof of Theorems \ref{intro-thm-noeth-mutation} and \ref{intro-thm-KS-mutation}]
By Lemma \ref{lem-at-most-two} and Proposition \ref{prop-2-term-at-most-two}, we have Theorem \ref{intro-thm-KS-mutation}.
If $\L$ is a module-finite $R$-algebra for a commutative noetherian ring $R$, then $\add P$ is functorially finite in $\cK$ for any complex $P\in\cK$.
Therefore Theorem \ref{intro-thm-noeth-mutation} (a) holds.
If $R$ is complete local, then Theorem \ref{intro-thm-noeth-mutation} follows from Theorem \ref{intro-thm-KS-mutation} (b).
\end{proof}
The following example gives an almost complete $2$-term silting complex which admits the Bongartz completion, but does not admit the co-Bongartz completion.
\begin{example}\label{example-not-noetherian}
Assume that there exists a primitive idempotent $e\in \L$ such that $\L e$ is not a finitely generated $e\L e$-module.
Then $\L(1-e)$ does not admit the co-Bongartz completion.
On the other hand, $\L(1-e)$ always admits the Bongartz completion, that is, $\L$.
We give one concrete algebra which fits into this situation as follows.

Let $K$ be a field and $Q$ be the following quiver
\begin{equation*}
\begin{tikzpicture}
\node(1)at(0,0){$1$};
\node(2)at(2,0){$2$};
\path[->, thick, font=\scriptsize ,>=angle 45]
(1.east) edge (2.west);
\draw (2) edge [->, font=\scriptsize ,>=angle 45, thick, out=30, in=330, looseness=6] (2);
\end{tikzpicture}
\end{equation*}
Let $\L:= \widehat{KQ}$ be the completed path algebra of $Q$, that is, $\L$ is a completion of the path algebra $KQ$ by the arrow ideal.
Then we have two indecomposable projective $\L$-modules $\L e_{1}$ and $\L e _{2}$.
Note that $KQe_i$ has $K$-basis the set of all paths starting from $i$.
Clearly, there exists a right $(\add \L e_{1})$-approximation of $\L[1]$ in $\cK$, but there exist no left $(\add \L e_{1})$-approximation of $\L$ in $\cK$.
Therefore, there exists only one $2$-term silting complex having $\L e_{1}$ as a direct summand, that is $\L$.
\end{example}

\subsection{Mutation of silting modules}\label{subsection-basic-supp-tau-tilt}
In this subsection, we assume that the homotopy category $\cK=\sK^{\rm b}(\proj\L)$ is Krull-Schmidt.
We remark the following lemma about the Krull-Schmidt property without a proof.
\begin{lem}\label{lem-krull-schmidt-proj-fp-idem}
Let $\L$ be a ring such that $\sK^{\rm b}(\proj \L)$ is Krull-Schmidt and $e\in\L$ be an idempotent.
Then the categories $\fp \L$, $\proj \L$ and $\sK^{\rm b}(\proj \L/\langle e \rangle)$ are also Krull-Schmidt.
\end{lem}
We study finitely generated silting modules defined as follows.
\begin{dfn}\label{dfn-silting-module}
\begin{itemize}
\item[(1)]
We say that a $\L$-module $M$ is a \emph{silting module} (resp. \emph{presilting module}) if there exists a $2$-term silting complex (resp. $2$-term presilting complex) $P$ in $\sK^{\rm b}(\proj\L)$ such that $\add M=\add H^0(P)$.
We denote by $\silt\L$ the set of isomorphism classes of basic silting $\L$-modules.
\item[(2)]
Let $M\in\fp\L$ and $P\in\proj\L$.
We say that a pair $(M, P)$ is a \emph{presilting pair} if $M$ is a presilting module and $\Hom_{\L}(P,M)=0$.
A presilting pair $(M,P)$ is a \emph{silting pair} if it satisfies the following two statements:
\begin{itemize}
\item
For each $Q\in\proj\L$, $\Hom_{\L}(Q,M)=0$ if and only if $Q\in\proj P$.
\item
There exists an exact sequence $\L \xto{f} M_{0} \to M_{1} \to 0$ with a left $(\add M)$-approximation $f$ and $M_{0}, M_{1}\in\add M$.
\end{itemize}
We say that a pair $(M,P)$ is \emph{basic} if $M$ and $P$ are basic, and say that $(M,P)$ is isomorphic to $(N,Q)$ if $M$ is isomorphic to $N$ and $P$ is isomorphic to $Q$.
\end{itemize}
\end{dfn}
Silting modules have been studied \cite{AngeleriHugel-Marks-Vitoria, Iyama-Jorgensen-Yang}, also known as support $\tau$-tilting modules \cite{Adachi-Iyama-Reiten}, see also \cite[Proposition 2.14]{Jasso}.

Since $\proj\L$ is Krull-Schmidt, $\L$ admits a finite indecomposable direct sum decomposition as a left $\L$-module.
Such a decomposition induces a pairwise orthogonal primitive idempotents $\{e_i \mid i\in I \}$ of $\L$ such that $1_{\L}=\sum_{i\in I}e_i$, where $I$ is a finite set.
We fix one such idempotents $\{e_i\mid i\in I\}$.
For each $\L$-module $M$, there exists a maximal subset $J$ of $I$ such that $e_jM=0$ for any $j\in J$.
In this case, we say that $e=\sum_{j\in J}e_j$ is the maximal idempotent (with respect to $\{e_i\mid i\in I\}$) with a property $eM=0$.

There exists a bijection between silting modules and $2$-term silting complexes, which was shown in  \cite{Adachi-Iyama-Reiten} in the case where $\L$ is a finite dimensional algebra.
Here we refer \cite{Iyama-Jorgensen-Yang}.
For a $2$-term complex $P$ in $\cK$, we denote by $\rho_{1}P$ the maximal direct summand of $P$ which belongs to $(\proj\L)[1]$.
\begin{thm}\label{thm-Tr-rigid-silting}\cite[Theorem 3.3]{Iyama-Jorgensen-Yang}
Let $\L$ be a ring such that $\sK^{\rm b}(\proj\L)$ is a Krull-Schmidt category.
\begin{itemize}
\item[(a)]
The map $P \mapsto (H^{0}P, (\rho_{1}P)[-1])$ gives a bijection from the first of the following sets to the second:
\begin{itemize}
\item[(i)]
The set of isomorphism classes of basic $2$-term presilting complexes.
\item[(ii)]
The set of isomorphism classes of basic presilting pairs.
\end{itemize}

\item[(b)]
The bijection in (a) restricts to a bijection from the first of the following sets to the second:
\begin{itemize}
\item[(i)]
The set of isomorphism classes of basic $2$-term silting complexes.
\item[(ii)]
The set of isomorphism classes of basic silting pairs.
\end{itemize}
Moreover, the map $(M,P)\mapsto M$ gives a bijection from the set (ii) to the following one.
\begin{itemize}
\item[(iii)]
The set of isomorphism classes of basic silting modules.
\end{itemize}
The map $M \mapsto Q\oplus (\L e[1])$ gives a bijection from (iii) to (i), where $Q$ is a minimal projective presentation of $M$ and $e\in\L$ is the maximal idempotent such that $eM=0$.
\end{itemize}
\end{thm}
Let $(M,P)$ and $(N,Q)$ be presilting pairs.
We say that $(M,P)$ is a direct summand of $(N,Q)$ if $M\in\add N$ and $P\in\add Q$ hold.
The following lemma is a module version of Lemma \ref{prop-presilt-silt-number}.
\begin{lem}\label{lem-dfn-sttilt-eq}
Let $(M,P)$ be a presilting pair such that there exists a silting pair having $(M,P)$ as a direct summand.
Then $(M,P)$ is a silting pair if and only if $|M|+|P|=|\L|$ holds.
\end{lem}
\begin{proof}
Let $S$ be a minimal projective presentation of $M$.
Then $|M|=|S|$ holds.
Therefore the assertion follows from Lemma \ref{prop-presilt-silt-number} and Theorem \ref{thm-Tr-rigid-silting}.
\end{proof}
Let $(M,P)$ be a presilting pair and $S$ the $2$-term presilting complex in $\cK$ corresponding $(M, P)$ by Theorem \ref{thm-Tr-rigid-silting}.
We say that $(M, P)$ is \emph{almost complete} if $|M|+|P|=|\L|-1$ holds.
This is equivalent to that $S$ is almost complete.
If there exists the Bongartz completion $T$ of $S$, then we call a silting pair corresponding to $T$ \emph{the Bongartz completion} of $(M,P)$.
Dually, \emph{the co-Bongartz completion} of $(M,P)$ is a silting pair corresponding to the co-Bongartz completion of $S$. 

We define mutation of silting pairs as follows.
\begin{dfn}
Let $(M,P), (N,Q)$ be silting pairs and $S, T$ the corresponding $2$-term silting complexes, respectively.
We say that $(N,Q)$ is \emph{an irreducible left mutation} of $(M,P)$ if $T$ is an irreducible left mutation of $S$.
\end{dfn}
By Theorem \ref{thm-Tr-rigid-silting} (b), there exists a bijection between the set $\silt\L$ of isomorphism classes of basic silting modules and the set of isomorphism classes of basic silting pairs, we define an irreducible left mutation of silting modules by using this bijection.

In general, it is difficult to calculate mutation of $2$-term silting complexes.
On the other hand, mutation of silting modules can be calculated as in Theorem \ref{thm-sttilt-mutation}.
We use the following lemma in the proof of Theorem \ref{thm-sttilt-mutation}.
\begin{lem}\label{lem-2-term-decompose}
Let $ P\in \sK^{\rm b}(\proj\L)=\cK$ be a $2$-term complex and $P^{\prime} \xto{\alpha} \L[1]$ a right $(\add P)$-approximation of $\L[1]$ in $\cK$.
Let $Q=C(\alpha)[-1]$, where $C(\alpha)$ is the cone of $\alpha$ in $\cK$.
Then we have $\add(\rho_{1}Q) \subset \add(\rho_{1}P)$.
\end{lem}
\begin{proof}
Since $\cK$ is a Krull-Schmidt category, there exists a decomposition 
\[
P = (X_{-1} \xto{d} X_0) \oplus (Y \to 0) \oplus (0 \to Z)
\]
in $\cK$ so that $d$ is left and right minimal in $\proj\L$.
We have $\rho_{1}P = (Y \to 0)$ because $d$ is right minimal.
Since $Q[1]=C(\alpha)$ is the cone of $\alpha$, $Q$ is isomorphic to the following $2$-term complex in $\cK$:
\[
X'_{-1}\oplus Y' \xto{\left[ \begin{smallmatrix} d' & 0 \\ f & g \end{smallmatrix} \right]} X'_0\oplus \L,
\]
where $(X'_{-1} \xto{d'} X'_0)\in\add (X_{-1} \xto{d} X_0)$, $Y'\in \add Y \subset \proj\L$ and $f : Y' \to X'_0$ and $g : Y' \to \L$.

Let $V$ be an indecomposable projective $\L$-module such that $(V\to 0)\in\add(\rho_1Q)$.
There exist morphisms $a_X \in \Hom_\L(V, X'_{-1})$, $a_Y \in \Hom_\L(V, Y')$, $b_X \in \Hom_\L(X'_{-1}, V)$ and $b_Y\in\Hom_\L(Y', V)$ which satisfy the following equations
\[
V \xto{\left[ \begin{smallmatrix} a_X \\ a_Y  \end{smallmatrix} \right]} X'_{-1}\oplus Y' \xto{\left[ \begin{smallmatrix} b_X & b_Y  \end{smallmatrix} \right]} V, \qquad 
{\rm id}_V=b_Xa_X + b_Ya_Y, \qquad \begin{bmatrix} d' & 0 \\ f & g \end{bmatrix} \begin{bmatrix} a_X \\ a_Y  \end{bmatrix} =0.
\]
If $b_Xa_X$ is an automorphism of $V$, then $(a_X, 0)$ is a morphism from $(V\to 0)$ to $(X'_{-1} \xto{d'} X'_0)$ in $\cK$ with a retraction $(b_1, 0)$.
Therefore $(V\to 0)\in \add (X_{-1} \xto{d} X_0)$ holds.
However this contradicts to $\rho_{1}(P) = (Y \to 0)$.
Assume that $b_Xa_X$ is not an automorphism of $V$.
Then $b_Ya_Y={\rm id}_V-b_Xa_X$ is an automorphism of $V$ since $V$ is indecomposable.
This implies that $(V\to 0)$ is a direct summand of $(Y \to 0)$ in $\cK$, so $(V \to 0)\in\add(\rho_1P)$ holds.
We have the assertion.
\end{proof}

We say that a $\L$-module $M$ is \emph{sincere} if for $P\in\proj\L$, $\Hom_{\L}(P,M)=0$ implies $P=0$.
The following proposition enables us to calculate mutation of silting modules.
\begin{thm}\label{thm-sttilt-mutation}
Assume that $\sK^{\rm b}(\proj\L)$ is a Krull-Schmidt category.
Let $(M, P)$ be a basic almost complete silting pair.
Assume that there exist the Bongartz completion and the co-Bongartz completion of $(M, P)$.
Then the following statements hold.
\begin{itemize}
\item[(a)]
The Bongartz completion of $(M, P)$ is of the form $(M\oplus N, P)$ for some indecomposable $N\in\fp\L$.
\item[(b)]
Let $h : N \to M'$ be a left minimal $(\add M)$-approximation of $N$.
Then the co-Bongartz completion of $(M, P)$ is of the form $(M\oplus \Cok h, P)$ if $\Cok h\neq 0$, or is of the form $(M, P\oplus P')$ for some $P'\in\proj\L$ if $\Cok h=0$.
\item[(c)]
Let $e\in\L$ be an idempotent such that $\add P = \add \L e$.
Then $\Cok h=0$ if and only if $M$ is not a sincere $\L/\langle e \rangle$-module.
\end{itemize}
\end{thm}
\begin{proof}
Let $S$ be a minimal projective presentation of $M$.
By Theorem \ref{thm-Tr-rigid-silting}, $S\oplus P[1]$ is an almost $2$-term presilting complex.
By our assumption, there exist the Bongartz completion $S\oplus P[1]\oplus Q$ of $S\oplus P[1]$ and the co-Bongartz completion $S\oplus P[1]\oplus U$ of $S\oplus P[1]$.
We may assume that $Q$ and $U$ are indecomposable.

(a) 
We show that $H^0Q\neq 0$.
Let $g$ be a right minimal $(\add S\oplus P[1])$-approximation of $\L[1]$, and let $Q' = C(g)[-1]$.
We have $Q\in\add Q'$.
By Lemma \ref{lem-2-term-decompose}, $\add(\rho_1Q^{\prime})\subset\add(\rho_1(S\oplus P[1]))=\add P[1]$ holds.
Suppose that $H^0Q=0$.
Since $Q$ is indecomposable, $\rho_1Q\simeq Q$ holds.
Thus we have $\add Q \subset\add P[1]$.
This is a contradiction since $S\oplus P[1]\oplus Q$ is the Bongartz completion of an alomost complete $2$-term presilting complex $S\oplus P[1]$.
So we have $H^0Q\neq 0$.
Therefore, for $N=H^0Q$, $(M\oplus N, P)$ is the Bongartz completion of $(M, P)$.

(b)
By Proposition \ref{prop-2-term-at-most-two}, there exists a triangle $Q \xto{f} S'\oplus P'[1] \to U \to Q[1]$ such that $f$ is a left minimal $(\add S\oplus P[1])$-approximation.
By taking the $0$-th cohomology, we have an exact sequence 
\begin{align}\label{mutation-sequence-tr-tilting}
H^{0}Q \xto{H^{0}f} H^0(S') \to H^{0}U \to 0.
\end{align}
Since $H^0U$ is indecomposable or zero, $H^0f$ is left minimal.
Since $f$ is a left $(\add S\oplus P[1])$-approximation, $H^{0}f$ is a left minimal $(\add M)$-approximation of $N=H^{0}Q$.
If $H^0U = \Cok(H^0f)\neq 0$, then $(M\oplus \Cok(H^0f), P)$ is the co-Bongartz completion of $(M, P)$.
If $H^0U=\Cok(H^0f)=0$, then $U\simeq \rho_1 U$ holds.
In this case $(M, P\oplus U[-1])$ is the co-Bongartz completion of $(M, P)$.

(c)
Let $e'\in\L$ be a primitive idempotent such that $e'\notin\ideal{e}$.
Then $e'M=0$ if and only if $\Hom_{\cK}(\L e', S)=0$ if and only if $S\oplus P[1] \oplus \L e'[1]$ is a silting complex.
By (b), this is equivalent to that $\Cok(H^0f)=0$.
\end{proof}
In the rest of this subsection, we collect some properties of presilting modules.
\begin{prop}\label{prop-Tr-rigid-Ext}
Let $M,N\in\mod\L$ and $P=(d : P_{-1}\to P_{0})$ (resp. $Q=(Q_{-1} \to Q_0)$) be a minimal projective presentation of $M$ (resp. $N$).
We regard $P, Q\in\sK^{\rm b}(\proj\L)$ as $2$-term complexes.
Then the following conditions are equivalent.
\begin{itemize}
\item[{\rm (a)}]
$\Hom_{\L}(d,N)$ is surjective.
\item[{\rm (b)}]
$\Hom_{\sK^{\rm b}(\proj\L)}(P, Q[1])=0$.
\end{itemize}
If moreover, $\L$ is a module-finite $R$-algebra for a commutative complete local noetherian ring $R$, then (a) and (b) are equivalent the following one.
\begin{itemize}
\item[{\rm (c)}]
$\Ext_{\L}^{1}(M,\Fac N)=0$  holds.
\end{itemize}
In particular, $M$ is a presilting module if and only if it satisfies one of (a) or (b) for $N=M$.
\end{prop}
\begin{proof}
The statements (a) and (b) are equivalent by a direct calculation.
We can prove that (a) is equivalent to (c) by the same way as \cite[Lemma 5.2]{Iyama-Jorgensen-Yang}.
\end{proof}
We say that $M$ is \emph{faithful} if the annihilator $\ann_{\L} M : =\{a\in\L \mid aM=0\}$ of $M$ is zero.
\begin{lem}\label{lem-sttilt-tilting-equivalence}
Let $M$ be a presilting $\L$-module and $e\in\L$ the maximal idempotent with a property $eM=0$.
Then the following statements hold.
\begin{itemize}
\item[(a)]
If $M$ is faithful, then the projective dimension of $M$ is at most one.
\item[(b)]
For any two-sided ideal $I\subset\L$, $M/IM$ is a presilting $\L/I$-module.
\item[(c)]
Assume that $M$ is a direct summand of a silting $\L$-module.
Then $M$ is a sincere silting $\L$-module if and only if $|M|=|\L|$ holds.
\end{itemize}
\end{lem}
\begin{proof}
(a)
Let $(d : P_{-1}\to P_{0})$ be a minimal projective presentation of $M$.
We show that $d$ is a monomorphism.
Since $M$ is faithful, there exists a monomorphism $f : P_{1} \to M^{\prime}$ with $M'\in\add M$.
Since $M$ is presilting, $f$ factors through $d$.
Therefore $d$ is a monomorphism and the projective dimension of $M$ is at most one.

(b)
Let $P=(d :P_{-1}\to P_{0})$ be a minimal projective presentation of $M$. 
This $P$ is a $2$-term presilting complex of $\L$.
By Lemma \ref{lem-two-term-fac} (b), $\overline{P}$ is a $2$-term presilting complex of $\L/I$.
Thus we have a silting $\L/I$-module $H^0(\overline{P}) = H^0(P)/IH^0(P) = M/IM$.

(c)
We have a presilting pair $(M, \L e)$.
By Lemma \ref{lem-dfn-sttilt-eq}, $M$ is a sincere silting $\L$-module if and only if $|M|+|\L e| =|\L|$ and $e=0$ hold.
Clearly, this is equivalent to $|M|=|\L|$.
\end{proof}
A finitely presented $\L$-module $M$ is called a \emph{partial tilting} $\L$-module if the projective dimension of $M$ is at most one and $\Ext_{\L}^{1}(M,M)=0$ holds.
We call $M$ a \emph{tilting module} if $M$ is partial tilting and there exists an exact sequence $0\to \L \to M_{0} \to M_{1} \to 0$ with $M_{0}, M_{1}\in\add M$.

We see relations between silting and tilting modules.
\begin{lem}\label{lem-silting-tilting}
Let $M\in\fp\L$ and $I=\ann M$.
The following statements hold.
\begin{itemize}
\item[(a)]
If $M$ is a presilting $\L$-module, then $M$ is a partial tilting $\L/I$-module.
\item[(b)]
If $M$ is a silting $\L$-module, then it is a tilting $\L/I$-module.
The converse holds if $M$ is a presilting $\L$-module.
\item[(c)]
$M$ is a faithful silting $\L$-module if and only if it is a tilting $\L$-module.
\end{itemize}
\end{lem}
\begin{proof}
(a)
By (a) and (b) of Lemma \ref{lem-sttilt-tilting-equivalence}, $M/IM=M$ is a presilting $\L/I$-module such that its projective dimension as a $\L/I$-module is at most one.
By Proposition \ref{prop-Tr-rigid-Ext}, $\Ext^1_{\L/I}(M, M)=0$ holds.
Therefore $M$ is a partial tilting $\L/I$-module.

(b)
Assume that $M$ is a silting $\L$-module.
By (a), $M$ is a partial tilting $\L/I$-module.
There exists an exact sequence $\L \xto{f} M_0 \to M_1 \to 0$, where $f$ is a left $(\add M)$-approximation of $\L$ and $M_0, M_1 \in\add M$.
Since $f$ is a left approximation, $\Ker f=\ann M$ holds.
Thus we have an exact sequence $0 \to \L/I \to M_0 \to M_1 \to 0$.
Therefore $M$ is a tilting $\L$-module.

Conversely, assume that $M$ is a presilting $\L$-module and is a tilting $\L/I$-module.
There exists an exact sequence $0\to \L/I \xto{f} M_{0} \to M_{1} \to 0$ with $M_{0}, M_{1}\in\add M$.
Since $M$ is a tilting $\L/I$-module, $f$ is a left $(\add M)$-approximation of $\L/I$.
Thus we have an exact sequence $\L \xto{f'} M_0 \to M_1 \to 0$, where $f'$ is the composite of the canonical surjection $\L \to \L/I$ and is a left $(\add M)$-approximation of $\L$.
Let $e\in\L$ be the maximal idempotent such that $eM=0$.
Then $(M, \L e)$ is a silting pair of $\L$, and thus $M$ is a silting $\L$-module.

(c)
The ``only if'' part follows from (b).
Clearly, tilting $\L$-module is a faithful silting $\L$-module.
\end{proof}
\subsection{Exchange quivers and Hasse quivers}\label{subsection-mutation-hasse-quiver}
In this subsection, under the assumption (M) as below, we investigate two quivers which are obtained from $\twosilt\L$.
One is the Hasse quiver $(\twosilt\L, \geq)$.
The other is an exchange quiver of $\twosilt\L$ which was introduced in \cite[Definition 2.41]{Aihara-Iyama}.
Throughout this subsection, we assume that $\sK^{\rm b}(\proj\L)$ is a Krull-Schmidt category, and $\L$ satisfies the following condition.
%
\begin{itemize}
\item[(M)]
For any $2$-term silting complex of $\sK^{\rm b}(\proj\L)$ and each indecomposable direct summand of it, an irreducible left or an irreducible right mutation exists.
\end{itemize}
For instance, if $R$ is complete local, then a module-finite $R$-algebra $\L$ satisfies the condition (M).
This condition (M) is equivalent to the following one: for each almost complete $2$-term silting complex $P\in\sK^{\rm b}(\proj\L)=\cK$, there exists a right $(\add P)$-approximation of $\L[1]$ in $\cK$ if and only if there exists a left $(\add P)$-approximation of $\L$ in $\cK$.

\begin{dfn}\label{def-stquiver}
We define an exchange quiver $Q(\twosilt \L)$ as follows.
\begin{itemize}
\item
The set of vertices is $\twosilt\L$.
\item
Draw an arrow from $S$ to $T$ if $T$ is an irreducible left mutation of $S$.
\end{itemize}
\end{dfn}
We see that the two quivers coincide.
\begin{thm}\label{thm-Hasse-mutation-quiver}
Assume that $\L$ satisfies the condition (M), and $\sK^{\rm b}(\proj\L)$ is a Krull-Schmidt category.
Then the exchange quiver $Q(\twosilt \L)$ is equal to the Hasse quiver $\mathsf{Hasse}(\twosilt \L, \, \geq)$:
\[
Q(\twosilt \L) = \mathsf{Hasse}(\twosilt \L, \, \geq).
\]
\end{thm}
Theorem \ref{thm-Hasse-mutation-quiver} is a direct consequence of the following proposition.
\begin{prop}\label{prop-irreducible-left-mutation}
Assume that $\L$ satisfies the condition (M).
Let $S,T\in\twosilt\L$ such that $S>T$.
Then there exists an irreducible left mutation $S^{\prime}$ of $S$ such that $S>S^{\prime}\geq T$ holds.
\end{prop}
\begin{proof}
Since $\L\geq S>S^{\prime}\geq T \geq \L[1]$ holds,
the assertion follows from \cite[Proposition 2.36]{Aihara-Iyama}.
Note that, an assumption (F) in \cite{Aihara-Iyama} is stronger than our condition (M).
However, since we deal with only $2$-term complexes, the condition (M) is enough to prove our assertion.
\end{proof}
\begin{proof}[Proof of Theorem \ref{thm-Hasse-mutation-quiver}]
Let $S,T\in\twosilt\L$ such that $S>T$.
We show that $T$ is an irreducible left mutation of $S$ if and only if there is no $U\in\twosilt\L$ such that $S>U>T$.
Assume that there is no $U\in\twosilt\L$ such that $S>U>T$.
By Proposition \ref{prop-irreducible-left-mutation}, there is an irreducible left mutation $S'$ of $S$ such that $S>S'\geq T$.
This implies that $S'=T$, that is, $T$ is an irreducible left mutation of $S$.

Conversely, assume that $T$ is an irreducible left mutation of $S$.
Suppose that there exists $U\in\twosilt\L$ such that $S > U > T$.
By Proposition \ref{prop-irreducible-left-mutation}, there exists $S'\in\twosilt\L$ such that $S>S'\geq U$ and $S'$ is an irreducible left mutation of $S$.
By \cite[Proposition2.19]{Aihara-Iyama} we have $\add S \cap \add T \subseteq \add S'$.
Since both of $T$ and $S'$ are irreducible left mutations of $S$, we have $\add T=\add S'$.
This is a contradiction.
Therefore there exists no $U$ such that $S > U > T$.
\end{proof}
We also define an exchange quiver $Q(\silt\L)$ of $\silt\L$ as follows, 
\begin{itemize}
\item
The set of vertices is $\silt\L$.
\item
Draw an arrow from $M$ to $N$ if $N$ is an irreducible left mutation of $M$.
\end{itemize}
By the definition, $Q(\silt\L)$ coincides with the exchange quiver of $\twosilt\L$.

We end this section by giving an example of an exchange quiver.
\begin{example}\label{example-triangular}
Let $(R, \mfm)$ be a commutative complete local noetherian ring and $\ell$ a non-negative integer.
Let
\begin{eqnarray*}
\L=\left[
\begin{array}{cc}
R & 0 \\
R/\mfm^{\ell} & R/\mfm^{\ell} \\
\end{array}
\right].
\end{eqnarray*}
This $\L$ has two indecomposable projective modules
$P_{1}=\begin{bmatrix}
R\\
R/\mfm^{\ell}
\end{bmatrix}$
and $P_{2}=\begin{bmatrix}
0\\
R/\mfm^{\ell}
\end{bmatrix}$.
We denote by $M_{1}$ a cokernel of the canonical inclusion map $P_{2}\to P_{1}$.
Then we have three non-trivial silting $\L$-modules $P_{1}\oplus M_{1}$, $P_{2}$, and $M_{1}$.
The exchange quiver $Q(\silt\L)$ is as follows
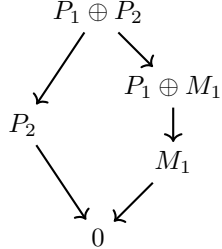
\begin{equation*}\label{example-triangular-mutation-quiver}
\begin{tikzpicture}
\node(L)at(0,0){$P_{1}\oplus P_{2}$};
\node(P1S1)at(1,-1){$P_{1}\oplus M_{1}$};
\node(P2)at(-1,-1.5){$P_{2}$};
\node(S1)at(1,-2){$M_{1}$};
\node(0)at(0,-3){$0$};

\draw[thick, ->] (L)--(P2);
\draw[thick, ->] (L)--(P1S1);
\draw[thick, ->] (P2)--(0);
\draw[thick, ->] (P1S1)--(S1);
\draw[thick, ->] (S1)--(0);

\end{tikzpicture}.
\end{equation*}
\end{example}
\section{Silting modules and torsion classes}\label{section-support-tau-tilting-and-torsion-class}
In this section, let $\L$ be a module-finite $R$-algebra, where $R$ is a commutative complete local noetherian ring.
Note that $\mod\L=\fp\L$ is a Krull-Schmidt abelian category.

We show that there exists a bijection between the set of silting modules and the set of functorially finite torsion classes of $\mod\L$, see Theorem \ref{thm-bij-sttilt-ftors}.

Let $\cA$ be an abelian category.
A pair $(\cT, \cF)$ of subcategories of $\cA$ is called a \emph{torsion pair} of $\cA$ if $\Hom_{\cA}(\cT, \cF)=0$ and there exists an exact sequence $0\to X^{\prime} \to X \to X^{\prime\prime} \to 0$ with $X^{\prime}\in\cT$ and $X^{\prime\prime}\in\cF$ for any $X\in\cA$.
If $(\cT,\cF)$ is a torsion pair, then we call $\cT$ a \emph{torsion class} of $\cA$ and $\cF$ a \emph{torsion free class} of $\cA$.
A subcategory $\cT$ of $\mod\L$ is a torsion class of $\mod\L$ if and only if $\cT$ is closed under factor modules and closed under extensions.

Presilting modules naturally induce torsion classes as follows.
\begin{lem}\label{lem-psilt-Fac}
For a presilting module $M$, $\Fac M$ is a torsion class of $\mod\L$.
\end{lem}
\begin{proof}
We show that $\Fac M$ is extension closed.
Let $0\to X\to Y \to Z \to 0$ be a short exact sequence with $X,Z\in\Fac M$.
There exists a surjection $M^{\oplus n}\to Z$.
By taking a pull back diagram of $Y \to Z \leftarrow M^{\oplus n}$, we have the following commutative diagram:
\[
\begin{tikzcd}
0 \ar[r] & X \ar[r] \ar[d, equal] & E \ar[r] \ar[d] & M^{\oplus n} \ar[r] \ar[d] & 0 \\
0 \ar[r] & X \ar[r] & Y \ar[r] & Z \ar[r] & 0,
\end{tikzcd}
\]
where each horizontal sequences are exact.
By Proposition \ref{prop-Tr-rigid-Ext}, the upper sequence splits and $E\simeq X\oplus M^{n}\in\Fac M$.
Since $M^{\oplus n}\to Z$ is surjective, so is $E\to Y$.
Therefore we have $Y\in\Fac M$.
\end{proof}
Let $\cT$ be a torsion class of $\mod\L$ and $M\in\cT$.
We say that $M$ is $\emph{Ext-projective}$ in $\cT$ if $\Ext_{\L}^{1}(M,-)|_{\cT}=0$ holds.
We denote by $P(\cT)$ the direct sum of one copy of each of the indecomposable Ext-projective objects in $\cT$ up to isomorphism.
We will see later that $P(\cT)$ is finitely generated if $\cT$ is functorially finite in $\mod\L$.



For any $M\in\mod\L$, let $M^{\perp_{1}}:=\{X\in\mod\L \mid \Ext_{\L}^{1}(M,X)=0\}$.

Since the following proposition is elementary we omit the proof.
\begin{prop}\label{prop-basic-tilting}
Let $T$ be a partial tilting $\L$-module.
\begin{itemize}
\item[(a)]
There exists $T^{\prime}\in\mod\L$ such that $T\oplus T^{\prime}$ is a tilting $\L$-module.
\item[(b)]
The following statements are equivalent.
\begin{itemize}
\item[(i)]
$T$ is a tilting $\L$-module.
\item[(ii)]
$\Fac T = T^{\perp_{1}}$ holds.
\item[(iii)]
For each $X\in T^{\perp_{1}}$, there exist an exact sequence $0\to Y \to T^{\prime} \to X \to 0$, where $T^{\prime}\in\add T$ and $Y\in T^{\perp_{1}}$.
\item[(iv)]
For each $X\in T^{\perp_{1}}$, $X$ is Ext-projective in $T^{\perp_{1}}$ if and only if $X\in\add T$.
\end{itemize}
\item[(c)]
$T$ is tilting if and only if $|T|=|\L|$ holds.
\end{itemize}
\end{prop}
We need the following lemma.
\begin{lem}\label{lem-Fac-minnimal}
Let $M\in\mod\L$.
If $\cT=\Fac M$ is a torsion class and each indecomposable direct summand $N$ of $M$ satisfies $N\not\in\Fac(M/N)$, then $M$ is Ext-projective in $\Fac M$.
\end{lem}
\begin{proof}
Since $\mod\L$ is a Krull-Schmidt category, we can prove this lemma by the same way as \cite[Chapter VI, Lemma 6.1]{ASS}.
\end{proof}
We denote by $\ann_{\L} M=\ann M=\{a\in\L \mid aM=0\}$ the annihilator of a $\L$-module $M$.
For a subcategory $\cC$ of $\mod\L$, let $\ann\cC:=\bigcap_{M\in\cC}\ann M$.
The following proposition observes a relationship between tilting modules and functorially finite torsion classes of $\mod\L$.
Note that each torsion class of $\mod\L$ is contravariantly finite in $\mod\L$.
In the case where $R$ is an artinian ring, the following proposition was shown in \cite{Smalo}.
\begin{prop}\label{prop-torsion-Fac-tilting}
Let $\cT$ be a torsion class of $\mod\L$ and $I=\ann\cT$.
Then the following conditions are equivalent.
\begin{itemize}
\item[{\rm (a)}]
$\cT$ is functorially finite in $\mod\L$.
\item[{\rm (b)}]
There exists a $\L$-module $M$ such that $\cT=\Fac M$.
\item[{\rm (c)}]
$P(\cT)$ is a tilting $\L/I$-module.
\end{itemize}
In this case, we have $\cT=\Fac(P(\cT))$.
\end{prop}
\begin{proof}
We show (a) implies (b).
Let $f: \L \to M$ be a left $\cT$-approximation of $\L$.
We claim that $\cT=\Fac M$.
Clearly, we have $\Fac M \subset\cT$.
Let $X\in\cT$ and take an epimorphism $g: \L^{\oplus n} \to X$.
Then $g$ factors through a morphism $f^{\oplus n}: \L^{\oplus n} \to M^{\oplus n}$.
Therefore $X\in\Fac M$.

We show that (b) implies (c).
Since $\cT=\Fac M$, we have $I=\ann\cT=\ann M$.
Since $\mod\L$ is a Krull-Schmidt category, we may assume that each indecomposable direct summand $N$ of $M$ satisfies $N\not\in\Fac(M/N)$.
By  Lemma \ref{lem-Fac-minnimal} (a), $\add M\subset \add P(\cT)$ holds.
Thus by Proposition \ref{prop-Tr-rigid-Ext}, $M$ is a presilting $\L$-module.
By Lemmas \ref{lem-sttilt-tilting-equivalence} (a) and \ref{lem-Fac-minnimal} (b), $M$ is a presilting $\L/I$-module such that the projective dimension is at most one.
Let $P=(P_{-1} \to P_0)$ be a minimal projective resolution of $M$ as a $\L/I$-module.
This $P$ is a $2$-term presilting complex of $\L/I$.
Let $f : \L/I \to P'$ be a left $(\add P)$-approximation of $\L/I$ in $\sK^{\rm b}(\proj \L/I)$.
We have the following exact sequence
\[
0=H^{-1}(P') \to H^{-1}(C(f)) \to \L/I \xto{H^0(f)} H^0(P') \to H^0(C(f)) \to 0,
\]
where $C(f)$ is the cone of $f$ in $\sK^{\rm b}(\proj \L/I)$.
Let $M'=H^0(P')\in\add M$. 
The map $H^0(f) : \L/I \to M'$ is a left $(\add M)$-approximation of $\L/I$.
Since $M$ is a faithful $\L/I$-module, $H^0(f)$ is injective.
Therefore we have $H^{-1}(C(f))=0$.
Let $M''=H^0(C(f))\in \Fac M$.
Then $M\oplus M'' = H^0(P\oplus C(f))$ is a faithful silting $\L/I$-module.
By Lemma \ref{lem-silting-tilting} (c), $M\oplus M''$ is a tilting $\L/I$-module.
By Proposition \ref{prop-basic-tilting} (b), $\add (M\oplus M'') = \add P(\cT)$ holds.
Therefore $P(\cT)$ is a tilting $\L/I$-module.

We show that (c) implies (a).
It is enough to show that $\cT$ is covariantly finite in $\mod\L$.
Put $M:=P(\cT)$.
There exists an exact sequence $0\to \L/I \xto{\iota} M^{\prime} \to M^{\prime\prime} \to 0$, where $M^{\prime}, M^{\prime\prime}\in\add M$.
Let $f:=\iota\circ\pi : \L \to M^{\prime}$, where $\pi : \L \to \L/I$ is a canonical morphism.
This $f$ is a left $\cT$-approximation of $\L$, since any morphism from $\L$ to an object of $\cT$ factors through $\pi$ and $M^{\prime\prime}$ is $\Ext$-projective in $\cT$.
For each $X\in\mod\L$, take an epimorphism $g: \L^{\oplus n} \to X$.
By taking a push out diagram of $X \leftarrow \L^{\oplus n} \to (M^{\prime})^{\oplus n}$, we have the following commutative diagram:
\[
\begin{tikzcd}
\L^{\oplus n} \ar[r, "f^{\oplus n}"] \ar[d, "g"] & (M^{\prime})^{\oplus n} \ar[d, "g^{\prime}"]\\
X \ar[r, "f^{\prime}"] & E. \\
\end{tikzcd}
\]
Since $g$ is an epimorphism, so is $g^{\prime}$.
Thus $E\in\cT$.
It is easy to see that $f^{\prime}$ is a left $\cT$-approximation of $X$.
\end{proof}
We denote by $\ftors\L$ the set of functorially finite torsion classes of $\mod\L$.
The following theorem follows from \cite[Theorem 5.1]{Iyama-Jorgensen-Yang} and by applying Lemma \ref{lem-silting-tilting} and Proposition \ref{prop-torsion-Fac-tilting}.
If $\L$ is a finite dimensional algebra, then the theorem was shown in \cite[Theorem 2.7]{Adachi-Iyama-Reiten}.
\begin{thm}\label{thm-bij-sttilt-ftors}
Let $\L$ be a module-finite $R$-algebra, where $(R, \mfm)$ is a commutative complete local noetherian ring.
There exists a bijection $$\silt\L \longrightarrow \ftors\L$$ given by $M\mapsto \Fac M$, and the inverse is given by $\cT\mapsto P(\cT)$.
\end{thm}
\begin{proof}
By Lemma \ref{lem-psilt-Fac} and Proposition \ref{prop-torsion-Fac-tilting}, the map $M\mapsto \Fac M$ is well-defined.
Let $\cT\in\ftors\L$.
By Proposition \ref{prop-torsion-Fac-tilting}, $\cT=\Fac(P(\cT))$ holds and $P(\cT)$ is a tilting $(\L/\ann P(\cT))$-module.
By Proposition \ref{prop-Tr-rigid-Ext}, $P(\cT)$ is a presilting $\L$-module.
Therefore by Lemma \ref{lem-silting-tilting} (b), $P(\cT)$ is a silting $\L$-module.

Let $M\in\silt\L$.
By Proposition \ref{prop-Tr-rigid-Ext}, $\add M\subset\add P(\Fac M)$ holds.
Since both $M$ and $P(\Fac M)$ are tilting $(\L/\ann M)$-modules, we have $|M|=|P(\Fac M)|$.
Therefore we have $\add M = \add P(\Fac M)$.
\end{proof}
By Theorems \ref{thm-Tr-rigid-silting} and \ref{thm-bij-sttilt-ftors}, we have bijections between $\twosilt\L$, $\silt\L$ and $\ftors\L$.
Recall that $\twosilt\L$ is a poset such that $P\geq Q$ if and only if $\Hom_{\sK^{\rm b}(\proj\L)}(P, Q[1])=0$. 
The set $\ftors\L$ is a poset by inclusion.
Thus $\silt\L$ has two partial orders which are induced from bijections.
We show that they are equal.
\begin{cor}\label{cor-2silt-msilt-ftors}
Let $P,Q\in\twosilt\L$.
Then $\Hom_{\sK^{\rm b}(\proj\L)}(P,Q[1])=0$ if and only if $\Fac(H^{0}P)\supset\Fac(H^{0}Q)$.
In particular, the following bijections are isomorphisms of posets.
\[
\twosilt\L \xto{H^0(-)} \silt\L \xto{\Fac(-)}  \ftors\L.
\]
\end{cor}
\begin{proof}
Assume that $\Hom_{\sK^{\rm b}(\proj\L)}(P,Q[1])=0$.
By Lemma \ref{lem-exist-triangle}, there exists a triangle $\L \xto{f} P' \to P'' \to \L[1]$ of $\sK^{\rm b}(\proj\L)$, where $f$ is a left $(\add P)$-approximation of $\L$.
Since $\Hom_{\sK^{\rm b}(\proj\L)}(P[-1],Q)=0$, any morphism from $\L$ to $H^0Q$ factors through $H^0f : \L \to H^0P'$.
In particular, a surjection $\L^{\oplus \ell} \to H^0Q$ factors through $(H^0f)^{\oplus \ell}$.
Thus $\Fac(H^{0}P)\supset\Fac(H^{0}Q)$ holds.

Let $M:=H^0P$, $N:=H^0Q$ and $P_{-1} \xto{d} P_0$ be a minimal projective resolution of $M$.
There exists an idempotent $e\in\L$ such that $P \simeq (P_{-1} \to P_0)\oplus (\L e [1])$.
By a direct calculation, $\Hom_{\sK^{\rm b}(\proj\L)}(P,Q[1])=0$ if and only if $\Hom_{\L}(d, N)$ is surjective and $e N=0$.
Assume that $\Fac M\supset\Fac N$ holds.
Since $\add M= \add P(\Fac M)$ holds by Theorem \ref{thm-bij-sttilt-ftors}, we have $\Ext_{\L}^1(M, \Fac N)=0$.
By Proposition \ref{prop-Tr-rigid-Ext}, we have that $\Hom_{\L}(d, N)$ is surjective and $e N=0$.

Therefore the partial order of $\silt\L$ induced from the bijection $H^0(-)$ coincides with the partial order induced from the bijection $\Fac(-)$.
Clearly, these bijections are isomorphisms of posets.
\end{proof}
\section{Reduction theorem}\label{section-reduction}
In this section, let $\L$ be a module-finite $R$-algebra, where $(R, \mfm)$ is a commutative local noetherian ring.
Let $I$ be an ideal of $\L$ such that $I\subseteq\mfm\L$.
We have $I \subset \mfm\L \subset \rad\L$ \cite[(5.22) Proposition]{Curtis-Reiner}.

We show that there exists a bijection from $\ftors \L$ to $\ftors(\L/I)$.
Recall that for a $2$-term complex $P=(P_{-1}\xto{d} P_{0})$ in $\cK(\L):=\sK^{\rm b}(\proj\L)$, we denote by
\[\overline{P}=(\L/I)\otimes_{\L}P=(P_{-1}/IP_{-1}\to P_{0}/IP_{0})\]
a $2$-term complex in $\cK(\L/I):=\sK^{\rm b}(\proj\L/I)$.
Then $H^{0}(\overline{P})=H^{0}(P)/IH^{0}(P)$ holds.

We begin with the following lemma.
\begin{lem}\label{lem-psi-kernel}
Let $P,Q\in\proj\L$ and $\psi : \Hom_{\L}(P,Q) \to \Hom_{\L/I}(P/IP, Q/IQ)$ be a natural morphism.
Then $\psi$ induces an isomorphism
\[
\Hom_{\L}(P,Q)/I \Hom_{\L}(P,Q)\simeq\Hom_{\L/I}(P/IP, Q/IQ).
\]
\end{lem}
\begin{proof}
The assertion is clear if $P=\L$ since $\psi$ is equal to a canonical morphism $Q \to Q/IQ$.
Therefore the assertion holds if $P=\L^{\oplus n}$ for an integer $n$.
For a general $P$, there is a projective $\L$-module $P'$ such that $P\oplus P'=\L^{\oplus n}$.
By taking a direct summand of $\Hom_{\L}(\L^{\oplus n},Q)$ we have the desired isomorphism.
\end{proof}
We have the following proposition, see also \cite[Theorem 4.1]{EJR}.
\begin{prop}\label{prop-twosiltL-twosiltLI}
Let $P,Q$ be $2$-term complexes in $\sK^{\rm b}(\proj \L)$.
\begin{itemize}
\item[(a)]
$\Hom_{\cK(\L)}(P,Q[1])=0$ if and only if $\Hom_{\cK(\L/I)}(\overline{P},\overline{Q}[1])=0$.
\item[(b)]
Assume that $\proj\L$ is a Krull-Schmidt category.
Then a map $P \mapsto \overline{P}$ defines a surjective map from $\twopsilt\L$ to $\twopsilt(\L/I)$.
Moreover, this map is restricted to a surjective map from $\twosilt\L$ to $\twosilt(\L/I)$.
\end{itemize}
\end{prop}
\begin{proof}
(a)
The ``only if'' part follows from Lemma \ref{lem-two-term-fac}.
We use the same diagram in the proof of the lemma, and write the diagram below for the convenience of readers.
Let $\overline{\L}:=\L/I$, $P=(P_{-1}\xto{\alpha} P_{0})$ and $Q=(Q_{-1}\xto{\beta} Q_{0})$.
We have a commutative diagram with exact rows
\[
\begin{tikzcd}
\Hom_{\L}(P_{-1}, Q_{-1}) \times \Hom_{\L}(P_{0},Q_{0}) \ar[r, "\gamma_{\alpha, \beta}"] \ar[d, "\chi"] & \Hom_{\L}(P_{-1},Q_{0}) \ar[r]\ar[d, "\psi"] & \Hom_{\cK(\L)}(P,Q[1]) \ar[r] \ar[d] & 0 \\
\Hom_{\overline{\L}}(\overline{P}_{-1}, \overline{Q}_{-1}) \times \Hom_{\overline{\L}}(\overline{P}_{0},\overline{Q}_{0})  \ar[r, "\gamma_{\overline{\alpha}, \overline{\beta}}"] & \Hom_{\overline{\L}}(\overline{P}_{-1},\overline{Q}_{0}) \ar[r] & \Hom_{\cK(\overline{\L})}(\overline{P},\overline{Q}[1]) \ar[r] & 0,
\end{tikzcd}
\]
where $\gamma_{\alpha, \beta}(f,g):=g\circ \alpha-\beta\circ f$ for $(f,g)\in\Hom_{\L}(P_{-1}, Q_{-1}) \times \Hom_{\L}(P_{0},Q_{0})$.
It is easy to see that $\psi$ is surjective.

We show the ``if'' part.
Assume that $\Hom_{\cK(\overline{\L})}(\overline{P},\overline{Q}$ $[1])=0$ holds.
Let $h\in\Hom_{\L}(P_{-1}, Q_{0})$.
Since $\gamma_{\overline{\alpha}, \overline{\beta}}$ is surjective, there exist $(f,g)\in\Hom_{\L}(P_{-1}, Q_{-1}) \times \Hom_{\L}(P_{0},Q_{0})$ such that $h - \gamma_{\alpha, \beta}(f,g)$ is in the kernel of $\psi$.
Thus, by Lemma \ref{lem-psi-kernel}, we have $\Hom_{\L}(P_{-1}, Q_{0})=\Im(\gamma_{\alpha, \beta})+I\Hom_{\L}(P_{-1}, Q_{0})$.
By Nakayama's lemma, $\Hom_{\L}(P_{-1}, Q_{0})=\Im(f_{\alpha, \beta})$ holds.
This means $\Hom_{\cK(\L)}(P,Q[1])=0$.

(b)
By (a), the map $\twopsilt\L\to\twopsilt(\L/I \L)$ is well-defined.
Since $\proj\L$ is a Krull-Schmidt category, any $\L$-morphism $P_{-1}/IP_{-1}$ $\to P_{0}/IP_{0}$ is lifted to a $\L$-morphism $P_{-1} \to P_{0}$.
Therefore again by (a), the map $\twopsilt\L\to\twopsilt(\L/I)$ is surjective.
This map restricts to a map $\twosilt\L\to\twosilt(\L/I)$, since $\thick P = \sK(\proj \L)$ implies $\thick \overline{P} = \sK(\proj \overline{\L})$ for any $P\in\twosilt\L$.
We show that this map surjective.
Let $S\in\twosilt(\L/I)$.
Since $\twopsilt\L\to\twopsilt(\L/I \L)$ is surjective, there exists $P\in\twopsilt\L$ such that $\overline{P}=S$.
Let $P\oplus Q\in\twosilt\L$ be the Bongartz completion of $P$.
Since both of $\overline{P}$ and $\overline{P}\oplus\overline{Q}$ are $2$-term silting complexes, we have $\add \overline{P} = \add (\overline{P}\oplus\overline{Q})$.
Therefore the map $\twosilt\L\to\twosilt(\L/I)$ is surjective.
\end{proof}
The main result of this section is the following theorem.
\begin{thm}\label{thm-reduction}
Let $\L$ be a module-finite $R$-algebra, where $(R, \mfm)$ is a commutative local noetherian ring.
Let $I$ be an ideal of $\L$ such that $I\subseteq \mfm\L$.

\begin{itemize}
	\item[(a)]
	We have an embedding of posets
	\[
	\twosilt\L \longrightarrow \twosilt(\L/I), \quad P \mapsto \overline{P}.
	\]
	
	\item[(b)]
	Assume that $(R, \mfm)$ is complete local.
	Then we have the following commutative diagram such that all maps are isomorphisms of posets:
	\begin{align}\label{diagram-twosilt-sttilt-ftors}
\begin{tikzpicture}
\node(twosiltL)at(0,2){$\twosilt\L$};
\node(twosiltLI)at(8,2){$\twosilt(\L/I)$};
\node(sttiltL)at(0,0){$\silt\L$};
\node(sttiltLI)at(8,0){$\silt(\L/I)$};
\node(ftorsL)at(0,-2){$\ftors\L$};
\node(ftorsLI)at(8,-2){$\ftors(\L/I).$};
\draw[thick, ->] (sttiltL)--(sttiltLI) node[midway, above]{$M\mapsto M/IM$} node[midway, below]{$\sim$};
\draw[thick, ->] (ftorsL)--(ftorsLI) node[midway, above]{$(-)\cap\mod(\L/I)$} node[midway, below]{$\sim$};
\draw[thick, ->] (twosiltL)--(twosiltLI) node[midway, above]{$P \mapsto \overline{P}$} node[midway, below]{$\sim$};
\draw[thick, ->] (sttiltL)--(ftorsL) node[midway, left]{$\Fac(-)$} node[midway, right]{$\wr$};
\draw[thick, ->] (sttiltLI)--(ftorsLI) node[midway, left]{$\Fac(-)$} node[midway, right]{$\wr$};
\draw[thick, ->] (twosiltL)--(sttiltL) node[midway, left]{$H^{0}(-)$} node[midway, right]{$\wr$};
\draw[thick, ->] (twosiltLI)--(sttiltLI) node[midway, left]{$H^{0}(-)$} node[midway, right]{$\wr$};
\end{tikzpicture}
\end{align}
	\end{itemize}
\end{thm}
\begin{proof}
(a)
By Proposition \ref{prop-twosiltL-twosiltLI}, the map $\overline{(-)}$ is a well-defined morphism of posets.
For $P, Q\in\twosilt\L$, by Proposition \ref{prop-twosiltL-twosiltLI} (a), $\overline{P}\geq \overline{Q}$ implies $P \geq Q$.
Thus $\overline{(-)}$ is an embedding of posets.

(b)
We show that the middle and the bottom horizontal map is well-defined and the diagram is commutative.
Since $H^0(\overline{P})=H^0P/IH^0P$ holds for any $P\in\twosilt\L$, the middle horizontal map is well-defined and is a morphism of posets.
Clearly, the upper square is commutative.
For $\cT\in\ftors\L$ and $X\in\mod(\L/I)$, if $X \to T$ is a left $\cT$-approximation of $X$, then the composite $X \to T \to T/IT$ is a left $(\cT\cap \mod(\L/I))$-approximation of $X$.
Thus the bottom horizontal map is well-defined.
Clearly, this map is a morphism of posets and the lower square is commutative.

By Proposition \ref{prop-twosiltL-twosiltLI} (b), the top map is an isomorphism of posets.
By Corollary \ref{cor-2silt-msilt-ftors}, all vertical maps are isomorphisms of posets.
Therefore all maps in the diagram are isomorphisms of posets.
\end{proof}
We remark that horizontal isomorphisms in (\ref{diagram-twosilt-sttilt-ftors}) preserve mutation.
\begin{cor}\label{cor-EJR-reduction}
The isomorphism $\silt\L \to \silt(\L/I),\, M\mapsto M/IM$ preserves mutation.
\end{cor}
\begin{proof}
Let $X,Y\in\cK(\L)$ be two complexes.
If $X \to Y^{\prime}$ is a left $(\add Y)$-approximation of $X$, then $\overline{X} \to \overline{Y^{\prime}}$ is a left $(\add \overline{Y})$-approximation of $\overline{X}$.
Therefore the map $\twosilt\L \to \twosilt(\L/I)$ preserves mutation.
\end{proof}
For a subcategory $\cX\subset \mod(\L/I)$, let
\begin{center}
$\pi^{-1}(\cX):=\{ X \in \mod \L \mid X/IX \in \cX\}$.
\end{center}
We use the following proposition in Section \ref{section-torsion-fl}.
\begin{prop}\label{prop-ftorsL-ftorsLI}
Let $\cT\in\ftors\L$.
We have 
\[
\cT=\pi^{-1}(\cT\cap\mod(\L/I)).
\]
\end{prop}
\begin{proof}
Let $\cT\in\ftors\L$.
Clearly, we have $\cT\subset\pi^{-1}(\cT\cap\mod(\L/I))$.
Let $M\in\silt\L$ such that $\Fac M=\cT$.
Take any $X\in\pi^{-1}(\cT\cap\mod(\L/I))$.
We have
\[
\frac{\Tr(M)\otimes_{\L}X}{I(\Tr(M)\otimes_{\L}X)}=\Tr(M)\otimes_{\L}\left(\frac{X}{IX}\right)=0,
\]
the last equality comes from that $X/IX\in\cT$, $M$ is $\Ext$-projective in $\cT$ and Proposition \ref{prop-Tr-rigid-Ext}.
By applying Nakayama's lemma for an $R$-module $\Tr(M)\otimes_{\L} X$, we have $\Tr(M)\otimes_{\L} X=0$.
Again by Proposition \ref{prop-Tr-rigid-Ext}, we have $\Ext_{\L}^{1}(M, \Fac X)=0$.
Consider the following exact sequence
\[
0 \to IX \to X \xto{p} X/IX \to 0.
\]
Since $I\subset\mfm\L$, $IX\in \Fac X$ holds.
Clearly we have $X/IX\in\cT\cap\Fac X$.
Take a surjective $\L$-homomorphism $f^{\prime} : M^{\prime} \to X/IX$, where $M^{\prime}\in\add M$.
Since $\Ext_{\L}^{1}(M, \Fac X)=0$, there exists a $\L$-homomorphism $f : M^{\prime} \to X$ which satisfies $p\circ f=f^{\prime}$.
Because $f^{\prime}$ is surjective, we have $\Im f+IX=X$.
By applying Nakayama's lemma to an $R$-module $X$, we have $\Im f=X$.
Therefore $X\in\Fac M =\cT$.
We have the assertion.
\end{proof}
A finite dimensional algebra $A$ is called \emph{$\tau$-tilting finite} if the number of support $\tau$-tilting modules is finite.
In this case, we have $\tors(A)=\ftors(A)$, see \cite{DIJ}.
The following corollary is useful when we study torsion classes of $\L$ such that $\L/\mfm\L$ is $\tau$-tilting finite.
\begin{cor}\label{cor-tau-tilting-finite}
Assume that $\L/\mfm\L$ is $\tau$-tilting finite.
Let $\cT\in\tors\L$ and $\cT_0=\pi^{-1}(\cT \cap \mod(\L/\mfm\L))$.
Then $\cT_0\in\ftors\L$ and $\cT_{0}\cap\mod(\L/\mfm\L)\subset\cT\subset\cT_{0}$ hold.
\end{cor}
\begin{proof}
Since $\L/\mfm\L$ is $\tau$-tilting finite, $\cT\cap\mod(\L/\mfm\L)\in \ftors(\L/\mfm\L)$ holds.
By Theorem \ref{thm-reduction}, there exists $\cT'\in\ftors\L$ such that $\cT'\cap\mod(\L/\mfm\L)=\cT\cap\mod(\L/\mfm\L)$.
Thus by Proposition \ref{prop-ftorsL-ftorsLI}, $\cT_0=\cT'\in\ftors\L$ holds.
We have
\[
\cT_{0}\cap\mod(\L/\mfm\L) = \cT\cap\mod(\L/\mfm\L) \subset \cT \subset \cT_0,
\]
where the inclusion $\cT \subset \cT_0$ follows from the definition of $\pi^{-1}$.
\end{proof}
\section{Torsion classes of module categories}\label{section-torsion-fl}
Throughout this section, let $R$ be a commutative noetherian ring and let $\L$ be a module-finite $R$-algebra.

We denote by $\Fl\L$ the category of finite length $\L$-modules.
When $(R, \mfm)$ is a local ring, we show that there exists a bijection from the set $\tors(\Fl\L)$ of all torsion classes of $\Fl\L$ to the set $\tors(\L/\mfm\L)$ of all torsion classes of $\mod(\L/\mfm\L)$ (Theorem \ref{thm-torsion-fl}).

First we observe bricks over $\L$.
A finitely generated $\L$-module $S$ is called \emph{brick} if its endomorphism algebra $\End_{\L}(S)$ is a division algebra.
We denote by $\brick(\L)$ the set of all bricks over $\L$.
We have the following lemma.
\begin{lem}\label{lem-brick}
Assume that $(R, \mfm)$ is a local ring.
Then any brick $S$ of $\L$ satisfies $\mfm S=0$.
In particular, for any two-sided ideal $I$ of $\L$ such that $I\subset \mfm\L$, we have $\brick(\L)=\brick(\L/I)=\brick(\L/\mfm\L)$.
\end{lem}
\begin{proof}
For any $r\in\mfm$, we have $rS\neq S$ by Nakayama's lemma.
Thus $(r\cdot)\in\End(S)$ is not an isomorphism.
Since $S$ is a brick, we have $(r\cdot)=0$.
Therefore $rS=0$.
\end{proof}
We use the following result.
\begin{prop}\cite[Corollary 5.20]{DIRRT}\label{thm-DIRRT}
Let $A$ be an Artin algebra and let $I$ be an ideal of $A$ contained in $\bigcap_{S\in\brick A}\ann S$. 
Then $\tors A$ and $\tors (A/I)$ are isomorphic as partially ordered sets, the map is given by $\cT\mapsto \cT\cap\mod(A/I)$.
\end{prop}
We need the following lemma.
\begin{lem}\label{lem-tors-L/I-L/Ii}
Assume that $(R, \mfm)$ is a local ring.
For any $i\geq 1$, the map $\cT \mapsto \cT\cap\mod(\L/\mfm\L)$ gives a bijection from $\tors(\L/\mfm^{i}\L)$ to $\tors(\L/\mfm\L)$.
\end{lem}
\begin{proof}
By Lemma \ref{lem-brick}, we have $\brick(\L/\mfm^i\L)=\brick(\L/\mfm \L)$.
Thus we apply Proposition \ref{thm-DIRRT} for $A=\L/\mfm^i \L$ and $I=\mfm\L/\mfm^i \L$, we have the assertion.
\end{proof}
For a full subcategory $\cC$ of $\mod\L$, we denote by $\sT_{\L}(\cC)$ the smallest torsion class of $\mod\L$ containing $\cC$.
We denote by $\Filt(\cC)$ the full subcategory of $\mod\L$ consisting of $M\in\mod\L$ such that there exists a sequence $0=M_{0} \subset M_{1} \subset \cdots \subset M_{n}=M$ of submodules of $M$ with $M_{i+1}/M_{i}\in\cC$.
Then we have $\sT_{\L}(\cC)=\Filt(\Fac \cC)$.

The first main result of this section is the following theorem.
\begin{thm}\label{thm-torsion-fl}
Let $(R, \mfm)$ be a commutative local noetherian ring and $\L$ a module-finite $R$-algebra.
We have a bijection 
\[
\tors(\Fl\L) \longrightarrow \tors(\L/\mfm\L)
\]
which is given by $\cT \mapsto \cT\cap\mod(\L/\mfm\L)$ and the inverse map is $\cT^{\prime} \mapsto \sT_{\L}(\cT^{\prime})$.
\end{thm}
\begin{proof}
Let $\Phi(\cT):=\cT\cap\mod(\L/\mfm\L)$ for $\cT\in\tors(\Fl\L)$.
Clearly, $\sT_{\L}(\cU)\in\tors(\Fl\L)$ holds for $\cU\in\tors(\L/\mfm\L)$.
Let $\cU\in\tors(\L/\mfm\L)$.
Since $\cU$ is a torsion class of $\mod(\L/\mfm\L)$, $\Fac \cU=\cU$ holds.
Since $\cT'$ is extension closed in $\mod(\L/\mfm\L)$, we have $\cU = \Filt\,\cU\cap\mod(\L/\mfm\L)$.
Therefore $\Phi \circ \sT_{\L}$ is an identity map of $\tors(\L/\mfm\L)$.

It remains to show that $\Phi$ is injective.
Let $\cT,\cU\in\tors(\Fl\L)$ such that $\Phi(\cT)=\Phi(\cU)$.
We have a commutative diagram
\begin{align*}
\begin{tikzpicture}
\node(fl)at(0,0){$\tors(\Fl\L)$};
\node(I)at(8,0){$\tors(\L/\mfm\L)$};
\node(Ii)at(4,-2){$\tors(\L/\mfm^{i}\L).$};
\draw[thick, ->] (fl)--(I) node[midway, above]{$\Phi=(-)\cap\mod(\L/\mfm\L)$};
\draw[thick, ->] (fl)--(Ii) node[midway, left=0.7cm]{$(-)\cap\mod(\L/ \mfm^{i}\L)$};
\draw[thick, ->] (Ii)--(I) node[midway, right=0.7cm]{$(-)\cap\mod(\L/\mfm\L)$};
\end{tikzpicture}
\end{align*}
Since the map $(-)\cap\mod(\L/\mfm\L) : \tors(\L/\mfm^{i}\L) \to \tors(\L/\mfm\L)$ is bijective by Lemma \ref{lem-tors-L/I-L/Ii}, we have $\cT\cap\mod(\L/\mfm^{i}\L)=\cU\cap\mod(\L/\mfm^{i}\L)$.
Since any object in $\Fl\L$ is contained in $\mod(\L/\mfm^i\L)$ for some $i$, we have $\cT=\cU$.
\end{proof}
We give one easy observation on Hasse quivers of torsion classes.
\begin{cor}\label{cor-torsion-hasse-subquvier}
The following statements hold.
\begin{itemize}
\item[(a)]
The Hasse quiver of $\tors(\Fl\L)$ is isomorphic to that of $\tors(\L/\mfm\L)$ as quivers.
\item[(b)]
The Hasse quiver of $\tors(\L/\mfm\L)$ is identified with a full subquiver of the Hasse quiver of $\tors\L$, where a vertex $\cT\in\tors(\L/\mfm\L)$ corresponds to $\sT_{\L}(\cT)\in\tors\L$.
\end{itemize}
\end{cor}
\begin{proof}
(a)
This follows from Theorem \ref{thm-torsion-fl}.

(b)
Since $\tors(\Fl\L)=\{\cT\in\tors\L \mid \cT \subseteq\Fl\L\}$, the assertion follows.
\end{proof}

From now on, we assume that $R$ is an integral domain wit Krull dimension one.
We denote by $K$ the fractional field of $R$.

Let $A:=K\otimes_{R}\L$, which is a finite dimensional $K$-algebra.
For a finitely generated $\L$-module $M$, let $KM:=K\otimes_{R}M \simeq A\otimes_{\L}M$.
We have an exact functor
$$K\otimes_{R}(-): \mod\L \to \mod A.$$

We study the relationship between torsion classes of $\mod\L$ and $\mod A$.
The following is a basic result of finitely generated modules.
Note that a finitely generated $\L$-module has finite length as a $\L$-module if and only if it has finite length as an $R$-module.
\begin{lem}\label{lem-local-global}
Let $M, N\in\mod\L$.
\begin{itemize}
\item[(a)]
For any $n\geq 0$, the functor $K\otimes_{R}(-) : \mod\L \to \mod A$ induces an isomorphism
$$
K\otimes_{R}\Ext^{n}_{\L}(M, N) \simeq \Ext_{ A}^{n}(KM, KN).
$$
In particular, for any $g\in\Ext_{ A}^{n}(KM, KN)$, there exist $f\in\Ext_{\L}^{n}(M,N)$ and $r\in R\setminus \{0\}$ such that $r^{-1}\otimes f=g$.

\item[(b)]
If $KM=0$ if and only if $M$ has finite length as a $\L$-module.

\item[(c)]
For any $X\in\mod A$, there exists $M\in\mod\L$ such that $KM \simeq X$.
\end{itemize}
\end{lem}
\begin{proof}
(a)
See \cite[(8.18) Corollary]{Curtis-Reiner}.

(b)
Both conditions are equivalent to that the support of $M$ consists of maximal ideals of $R$.

(c)
See \cite[(23.13) Proposition]{Curtis-Reiner}.
\end{proof}
For a full subcategory $\cC$ of $\mod\L$, let \[K\cC:=\{ KM \mid M\in\cC \}\] be a subcategory of $\mod A$.
\begin{prop}\label{prop-torsion-L-tildaL-map}
The following statements hold.
\begin{itemize}
\item[(a)]
The assignment $\cT \mapsto K\cT$ gives a surjective map
$$
K(-) : \tors\L \longrightarrow \tors A.
$$
\item[(b)]
For $\cT\in\tors\L$, $K\cT=0$ if and only if $\cT$ is contained in $\Fl\L$.
\end{itemize}
\end{prop}
\begin{proof}
(a)
Let $\cT$ be a torsion class of $\mod\L$.
We first show that $K\cT$ is a torsion class of $\mod A$.
Let $g : KT \to X$ be a surjection with $T\in\cT$ and $X\in\mod A$.
By Lemma \ref{lem-local-global} (c), there exists $M\in\mod \L$ such that $KM\simeq X$.
By Lemma \ref{lem-local-global} (a), there exist $f : T \to M$ and $r\in R\setminus \{0\}$ such that ${\rm id}_{K}\otimes f= rg$.
We have $KM\simeq K\Im(f) \in K\cT$.
Thus $K\cT$ is closed under factor modules.
By Lemma \ref{lem-local-global} (a), $K\cT$ is closed under taking extensions.
Therefore the map is well-defined.

For $\cU\in\tors A$, let $\cT = \{ M \in \mod\L \mid KM \in \cU \}$.
It is easy to see that $\cT$ is  a torsion class of $\mod\L$.
By Lemma \ref{lem-local-global} (c), $K\cT = \cU$ holds.

(b) The assertion directly follows from Lemma \ref{lem-local-global} (b).
\end{proof}
We observe that the map $K(-) : \tors\L \to \tors A$ is injective if it is restricted to a certain subset of $\tors \L$.
For two subcategories $\cA, \cB$ of $\mod\L$, let $[\cA, \cB]_{\L}:=\{ \cC\in \tors\L \mid \cA \subseteq \cC \subseteq \cB \}$.
If there is no danger of confusion, we write $[\cA, \cB]_{\L}=[\cA, \cB]$.
For a torsion class $\cT$ of $\mod\L$, let 
\[
\overline{\cT} :=\{ M \in \mod\L \mid \Fac M \cap \Fl\L \subseteq \cT\}.
\]
We have $\cT \cap \Fl\L = \overline{\cT} \cap \Fl\L$.
\begin{prop}\label{prop-injectivity-localizition-map}
\begin{itemize}
\item[(a)]
Let $\cB, \cC\in\tors\L$ such that $\cB \cap \Fl\L = \cC \cap \Fl\L$.
If $K\cB = K\cC$ holds, then $\cB = \cC$.

\item[(b)]
For each $\cT\in\tors\L$, the map $K(-) : \bigl[\cT \cap \Fl\L, \overline{\cT}\bigr] \to \tors A$ is injective.
In particular, we have a bijection $$K(-) : \bigl[\Fl\L, \mod\L\bigr] \longrightarrow \tors A.$$
\end{itemize}
\end{prop}
\begin{proof}
(a)
For $B \in \cB$, let $f : C \to B$ be a right $\cC$-approximation of $B$.
By Lemma \ref{lem-local-global} (a), ${\rm id}_{K}\otimes f$ is a right $K\cC$-approximation of $KB$.
Since $KB\in K\cC$, ${\rm id}_{K}\otimes f$ is split epimorphism and $\Cok(f)$ has finite length.
We have an exact sequence $0 \to \Im(f) \to B \to \Cok(f) \to 0$ with $\Im(f) \in \cC$ and $\Cok(f) \in \cB \cap \Fl\L \subset \cC$.
Therefore $B\in\cC$.

(b)
The assertion follows from (a) and Proposition \ref{prop-torsion-L-tildaL-map}.
\end{proof}
Now we are able to show the second main result of this section.
\begin{thm}\label{thm-torsion-disjoint-union}
Let $(R,\mfm)$ be a local noetherian integral domain with Krull dimension one and $K$ the fractional field of $R$.
Let $A=K\otimes_R \L$.
Then the following statements hold.
\begin{itemize}
\item[(a)]
For each $\cT\in\tors(\Fl\L)$, we have a bijection
\[
K(-): \bigl[\cT,\, {\cT}\bigr]_{\L} \longrightarrow \bigl[0,\, K(\overline{\cT})\bigr]_A.
\]
\item[(b)]
We have
\[
\tors\L = \coprod_{\cT\in\tors(\Fl\L)}\,[\cT, \, \overline{\cT}]_{\L}.
\]
\item[(c)]
Consequently we have a bijection
\[
\coprod K(-) : \tors\L = \coprod_{\cT\in\tors(\Fl\L)}\,[\cT, \, \overline{\cT}]_{\L} \longrightarrow \coprod_{\cT\in\tors(\Fl\L)}\,[0, \, K(\overline{\cT})]_A.
\]
\end{itemize}
\end{thm}
\begin{proof}
(a)
By Proposition \ref{prop-injectivity-localizition-map} (b), the map $K(-)$ is injective.
For a torsion class $\cC\subset\mod A$ such that $\cC\in \bigl[0, K(\overline{\cT})\bigr]_A$, let $\mathcal{U}$ be the full subcategory of $\overline{\cT}$ consisting of $X$ satisfying $KX\in\cC$.
Then it is easy to see that $\mathcal{U}\in\tors\L$, $\mathcal{U} \in \bigl[\cT, \overline{\cT}\bigr]_{\L}$ and $K\mathcal{U}=\cC$ hold.

(b)
For $\cC\in\tors\L$, let $\cT=\cC\cap\Fl\L$.
Then we have $\cC\in [\cT, \overline{\cT}]_{\L}$.
For any $\cT\in\tors(\Fl\L)$ and any $\cC\in [\cT, \overline{\cT}]_{\L}$, we have $\cC \cap \Fl\L = \cT$.
Thus $[\cT, \overline{\cT}]_{\L} \cap [\cU, \overline{\cU}]_{\L} = \emptyset$ if $\cT \neq \cU$.
Therefore the equality holds.

(c)
The map is bijection by (a).
\end{proof}
We have the following corollary.
\begin{cor}\label{cor-finiteness}
Assume that the same assumption as Theorem \ref{thm-torsion-disjoint-union} holds.
Moreover, assume that $(R, \mfm)$ is a complete local ring.
Then the set $\tors\L$ is finite if and only if so are $\tors(\L/\mfm\L)$ and $\tors A$.
\end{cor}
\begin{proof}
By Theorem \ref{thm-torsion-fl}, $\tors(\L/\mfm\L)$ is finite if and only if $\tors(\Fl\L)$ is finite.
By Proposition \ref{prop-injectivity-localizition-map} (b), $\tors A$ is finite if and only if the interval $[\Fl\L, \mod\L]_{\L}$ is finite.
Thus ``only if'' part holds.
By Theorem \ref{thm-torsion-disjoint-union} (c), ``if'' part holds.
\end{proof}
In general it is hard to calculate $\overline{\cT}$.
If $\L/\mfm\L$ is $\tau$-tilting finite, then the situation is easy.
\begin{cor}\label{cor-finiteness-criterion}
Assume that the same assumption as Theorem \ref{thm-torsion-disjoint-union} holds.
Moreover, assume that $(R, \mfm)$ is a complete local ring and $\L/\mfm\L$ is a $\tau$-tilting finite algebra.
Then the following statements hold.
\begin{itemize}
	\item[(a)]
	The following equality and a bijection
	\[
	\tors\L = \coprod_{M\in\silt\L}\,\bigl[\Fac M\cap\Fl\L, \, \Fac M\bigr]_{\L} \longrightarrow \coprod_{M\in\silt\L}\bigl[0, \Fac(KM)\bigr]_A.
	\]

	\item[(b)]
	If $A$ is Morita equivalent to a local algebra, then we have
	\[
	\tors\L = \coprod_{M\in\silt\L \cap \Fl\L}\bigl\{ \Fac M \bigr\} \sqcup \coprod_{N\in\silt\L \setminus\Fl\L}\bigl\{ \Fac N\cap\Fl\L , \, \Fac N \bigr\}.
	\]
	
	\item[(c)]
	Assume that $A$ is Morita equivalent to a local algebra.
	Let $Q$ be the Hasse quiver of $\tors(\Fl\L)$ and $\Gamma$ the full subquiver of $Q$ such that $\Gamma_0=\{ \Fac N \cap \Fl\L \mid N \in\silt\L \setminus \Fl\L\}$.
	We denote by $\iota : \Gamma \to Q$ the canonical inclusion map of quivers.
	Then the Hasse quiver $H$ of $\tors\L$ is given by $H_0 = Q_0 \sqcup \Gamma_0$ and $H_1 = Q_1 \sqcup \Gamma_1 \sqcup \{ x \to \iota(x) \mid x \in \Gamma_0\}$.
\end{itemize}
\end{cor}
\begin{proof}
(a)
Let $M\in\silt\L$.
It is easy to see that $K\Fac M = \Fac(KM)$ holds.
By Proposition \ref{prop-ftorsL-ftorsLI}, $\overline{\Fac M \cap \Fl\L}=\Fac M$ holds.
Thus the assertion follows from Theorem \ref{thm-torsion-disjoint-union}.

(b)
For any $M\in\silt\L$, $\overline{\Fac M \cap \Fl\L} = \Fac M$ holds.
For $M\in\silt\L \cap \Fl\L$, we have $[\Fac M \cap \Fl\L, \Fac M]_{\L}=\{\Fac M\}$.
For $N\in\silt\L \setminus\Fl\L$, $\Fac N \cap \Fl\L \subsetneq \Fac N$ and $\Fac KN \neq 0$ hold.
Since $A$ is Morita equivalent to a local algebra, $\tors A = \{0, \mod A\}$ holds.
Thus $[\Fac N \cap \Fl\L, \Fac N]_{\L}=\{\Fac N \cap \Fl\L, \Fac N\}$ holds.
Then the equation follows from (a).

(c)
We identify $\Gamma_0 = \{\Fac N  \mid N \in \silt\L \setminus \Fl\L \}$ by a bijection $\Fac N \mapsto \Fac N \cap \Fl\L$.
Then we have $H_0=Q_0\sqcup \Gamma_0$ by (b).
Let $x, y\in H_0$.
We have the following three cases:
(i) $x, y\in Q_0$, (ii) $x, y \in \Gamma_0$, and (iii) $x \in \Gamma_0, y\in Q_0$.

(i)
Assume that $x, y \in Q_0$.
Clearly, there exists an arrow $x \to y$ in $Q$ if and only if there exists an arrow $x \to y$ in the Hasse quiver of $\tors\L$.

(ii)
Assume that $x, y \in \Gamma_0$.
Let $x=\Fac N$ and $y=\Fac M$ for $M, N \in\silt\L\setminus\Fl\L$.
By Proposition \ref{prop-ftorsL-ftorsLI}, $[\Fac M, \Fac N]=\{\Fac M, \Fac N\}$ holds if and only if $[\Fac M \cap \Fl\L, \Fac N\cap \Fl\L]=\{\Fac M\cap \Fl\L, \Fac N\cap \Fl\L\}$ holds.
This implies that there exists an arrow $x \to y$ in the Hasse quiver of $\tors\L$ if and only if there exists an arrow $\iota(x) \to \iota(y)$ in $Q$.
Moreover, this is equivalent to the existence of an arrow $x \to y$ in $\Gamma$.

(iii)
Assume that $x = \Fac N \in \Gamma_0$ and $y\in Q_0$ for $N\in\silt\L\setminus\Fl\L$.
If there exists an arrow $x \to y$ in the Hasse quiver of $\tors\L$, then we have $y \subseteq \Fac N \cap \Fl\L \subsetneq \Fac N = x$.
Thus $y = \Fac N \cap \Fl\L = \iota(x)$ holds.

Therefore we have $H_1 = Q_1 \sqcup \Gamma_1 \sqcup \{ x \to \iota(x) \mid x \in \Gamma_0\}$
\end{proof}

\section{Examples of silting modules and torsion classes}\label{section-example}
Throughout this subsection, let $R=k[[x]]$ be the ring of formal power series in one variable with a field $k$.
We denote by $\mfm=(x)$ the maximal ideal of $R$.
In this section, we give some examples of silting modules and torsion classes of a module-finite $R$-algebras.

An $R$-algebra $\L$ is called an \emph{$R$-order} if $\L$ is a finitely generated projective $R$-module.
A \emph{$\L$-lattice} is a $\L$-module which is a finitely generated projective $R$-module.
We denote by $\CM\L$ the full subcategory of $\mod\L$ consisting of $\L$-lattices.
It is easy to see that $(\Fl\L, \CM\L)$ is a torsion pair of $\mod\L$.

\subsection{A hereditary order}\label{subsection-hereditary-torsionclass}
Let $\L$ be an $R$-algebra defined by the following $(n\times n)$-matrix form $(n\geq 1)$:
\begin{eqnarray}\label{def-L-hereditary-order}
	\L=\left[\begin{array}{ccccc}
		R & R & & \cdots & R \\
		\mfm & R & & & \vdots \\
		& \mfm & \ddots & &  \\
		\vdots & & \ddots  & R & R \\
		\mfm & \cdots & & \mfm & R
	\end{array}
	\right] \subset \mathrm{M}_n(R).
\end{eqnarray}
In this subsection, we study $\tors \L$ the set of all torsion classes of $\mod\L$.

Let $Q$ be the following quiver
\begin{equation*}
	\begin{tikzpicture}
	\node(1)at(0,0){$1$};
	\node(2)at(2,0){$2$};
	\node(cdots)at(4,0){$\cdots$};
	\node(n-1)at(6,0){$n-1$};
	\node(n)at(8,0){$n$};
	\draw[thick, ->] (1)--(2)node[midway, above]{$a_{1}$};
	\draw[thick, ->] (2)--(cdots)node[midway, above]{$a_{2}$};
	\draw[thick, ->] (cdots)--(n-1)node[midway, above]{$a_{n-2}$};
	\draw[thick, ->] (n-1)--(n)node[midway, above]{$a_{n-1}$};
	\draw[thick, ->] (n)--(8,-1)--node[above]{$a_{n}$}(0,-1)--(1);
	\end{tikzpicture}.
\end{equation*}
Then we have an isomorphism
\begin{align}\label{isom-kQ-hered}
\widehat{kQ} \simeq \L
\end{align}
for the complete path algebra $\widehat{kQ}$ of $Q$.
For $1 \leq i \leq n$, the path $e_i$ of length zero corresponds to the matrix unit whose $(i, i)$ entry is $1_R$.
For $1 \leq i \leq n-1$, the arrow $a_i$ corresponds to the matrix unit whose $(i, i+1)$ entry is $1_R$, and the arrow $a_n$ corresponds to the matrix unit whose $(n, 1)$ entry is $x$.

For $j\geq 0$, we denote by $Q_j$ the set of all paths of length $j$.
For an integer $i\geq 0$, the isomorphism (\ref{isom-kQ-hered}) induces an isomorphism
\begin{align*}
	\L/\mfm^{i}\L \simeq kQ/\langle Q_{ni} \rangle.
\end{align*}

The algebra $\L/\mfm\L$ is a Nakayama algebra and therefore is representation finite (see, \cite[Chapter V]{ASS}).
In particular, $\L/\mfm\L$ is $\tau$-tilting finite, and hence $\ftors(\L/\mfm\L)=\tors(\L/\mfm\L)$ holds \cite{DIJ}.

The $\tau$-tilting theory of Nakayama algebras was studied by Adachi \cite{Adachi}, and
here we recall his results.
Let $\Gamma$ be a module-finite $R$-algebra or a finite dimensional algebra, where $R$ is  a commutative complete local noetherian ring or a field.
A silting $\Gamma$-module $M$ is sincere if and only if $|M|=|\L|$ holds.
Let $\tautilt \Gamma$ be the set of isomorphism classes of basic sincere silting $\Gamma$-modules and $\left( \tautilt \Gamma \right)^{\rm c}=\silt\Gamma \setminus \tautilt \Gamma$. 
We refer \cite[Theorem 2.6, Proposition 2.8 and Corollary 2.29]{Adachi} for the following proposition.
\begin{prop}\cite{Adachi}\label{prop-adachi-nakayama}
For an algebra $kQ/\langle Q_{n} \rangle$, the following statements hold.
\begin{itemize}
\item[(a)]
A silting $kQ/\langle Q_{n} \rangle$-module is sincere if and only if it has a projective module as a direct summand.
\item[(b)]
There exists a bijection between $\tautilt(kQ/\langle Q_{n} \rangle)$ and $\left( \tautilt (kQ/\langle Q_{n} \rangle) \right)^{\rm c}$.
\item[(c)]
We have
\[
| \silt (kQ/\langle Q_{n} \rangle) | = \begin{pmatrix}
2n\\
n
\end{pmatrix}.
\]
\end{itemize}
\end{prop}
The $R$-algebra $\L$ of (\ref{def-L-hereditary-order}) is an $R$-order and is known to be a \emph{hereditary order}, that is, any left $\L$-ideal is $\L$-projective.
About hereditary orders, the following lemma is well-known.
\begin{lem}\label{lem-hereditary-order}
Let $\Gamma$ be a hereditary $R$-order.
Then we have the following properties.
\begin{itemize}
\item[(a)]
$\CM\Gamma=\proj\Gamma$ holds.
\item[(b)]
Each indecomposable module in $\mod \Gamma$ is either finite length or projective.
\end{itemize}
\end{lem}
\begin{proof}
(a)
This is well-known \cite{Curtis-Reiner} (see also \cite[1.6 Theorem]{Hijikata-Nishida}).

(b)
Since $\CM\Gamma=\proj\Gamma$, the torsion pair $(\Fl\Gamma, \CM\Gamma)$ splits.
\end{proof}
\begin{lem}\label{lem-tau-tilting-infinite-length}
Let $\L$ be an algebra as in (\ref{def-L-hereditary-order}), and $M$ a silting $\L$-module.
Then the following statements are equivalent.
\begin{itemize}
\item[(i)]
$M$ is sincere.
\item[(ii)]
$M$ has a projective $\L$-module as a direct summand.
\item[(iii)]
$M$ is an infinite length module.
\end{itemize}
\end{lem}
\begin{proof}
We show (i) implies (ii).
Assume that $M$ is a sincere silting $\L$-module.
Then $M/\mfm M$ is a sincere silting $\L/\mfm \L$-module.
By Proposition \ref{prop-adachi-nakayama} (a), there exists an idempotent $e\in\L$ such that $(\L/\mfm \L)e$ is a direct summand of $M/\mfm M$.
By Proposition \ref{prop-ftorsL-ftorsLI}, $\L e \in \pi^{-1}(\Fac M \cap \mod(\L/\mfm \L))=\Fac M$ holds.
Therefore (ii) holds.
The statement (ii) implies (i), since each indecomposable projective $\L$-module is sincere.
By Lemma \ref{lem-hereditary-order}, (ii) is equivalent to (iii).
\end{proof}
Then we classify all torsion classes of $\mod \L$.
\begin{thm}\label{thm-tors-hereditary}
Let $\L$ be an algebra as in (\ref{def-L-hereditary-order}).
\begin{itemize}
\item[(a)]
We have $\silt\L \simeq \silt(kQ/\ideal{Q_n})$.
In particular, we have
\[
| \silt \L | = \begin{pmatrix}
2n\\
n
\end{pmatrix}.
\]

\item[(b)]
Let $\cT\in\tors \L$.
Then $\cT$ satisfies precisely one of the following statements.
\begin{itemize}
\item[(i)]
$\cT=\Fac M$ for some $M \in (\tautilt\L)^{\rm c}$.
\item[(ii)]
$\cT=\Fac N \cap \Fl\L$ for some $N \in \tautilt\L$.
\item[(iii)]
$\cT=\Fac N$ for some $N \in \tautilt\L$.
\end{itemize}
Namely, we have the following equality
$$
\tors\L = \coprod_{M\in(\tautilt\L)^{\rm c}}\bigl\{ \Fac M \bigr\} \sqcup \coprod_{N\in\tautilt\L}\bigl\{ \Fac N \cap\Fl\L , \, \Fac N \bigr\}.
$$
In particular, we have
\[
| \tors \L | =\frac{3}{2} \begin{pmatrix}
2n\\
n
\end{pmatrix}.
\]
\end{itemize}
\end{thm}
\begin{proof}
(a)
The claim follows from Theorem \ref{thm-reduction} (b) and Proposition \ref{prop-adachi-nakayama}.

(b)
Let $K$ be a fractional field of $R$.
Then $A:=K\otimes_{R} \L \simeq M_{n}(K)$ is a simple algebra.
Let $M$ be a silting $\L$-module.
By Lemma \ref{lem-tau-tilting-infinite-length}, $M$ is sincere if and only if the length of $M$ is infinite.
Therefore,  by applying Corollary \ref{cor-finiteness-criterion} (b), we have the equality about $\tors\L$.

The last assertion follows from Proposition \ref{prop-adachi-nakayama} (b), and (c).
\end{proof}
\begin{example}
Let $n=2$ and
\[
\L = \left[\begin{array}{cc}
R & R \\
\mfm & R
\end{array}
\right].
\]
We denote by $S_{i}$ a simple module with a projective cover $\L e_{i}$ for $i=1,2$.
The exchange quiver $Q(\silt\L)$ and the Hasse quiver $H$ of $\tors\L$ are described as follows.
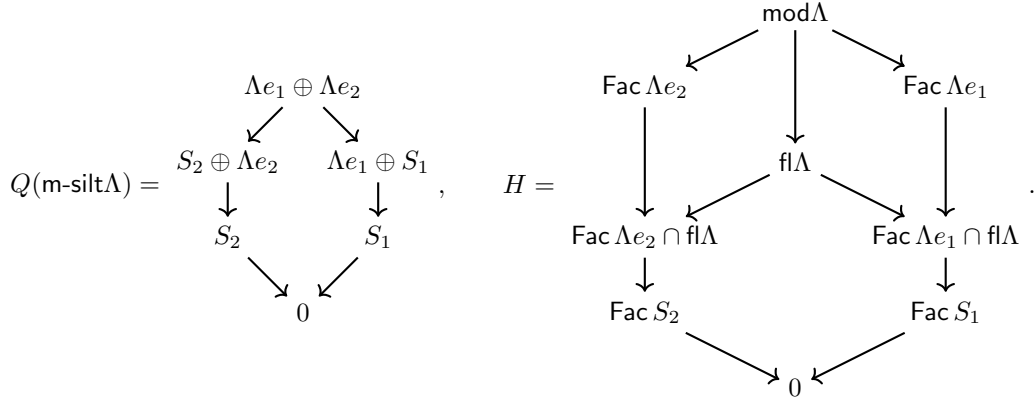
\begin{equation*}\label{example-triangular-mutation-quiver}
Q(\silt\L)=\begin{tikzpicture}[baseline=-40]
\node(L)at(0,0){$\L e_{1}\oplus \L e_{2}$};
\node(M2P2)at(-1,-1){$S_{2}\oplus \L e_{2}$};
\node(P1M1)at(1,-1){$\L e_{1}\oplus S_{1}$};
\node(M2)at(-1,-2){$S_{2}$};
\node(M1)at(1,-2){$S_{1}$};
\node(0)at(0,-3){$0$};
\draw[thick, ->] (L)--(M2P2);
\draw[thick, ->] (L)--(P1M1);
\draw[thick, ->] (M2P2)--(M2);
\draw[thick, ->] (P1M1)--(M1);
\draw[thick, ->] (M2)--(0);
\draw[thick, ->] (M1)--(0);
\end{tikzpicture},
\qquad
H=\begin{tikzpicture}[baseline=-40]
\node(mod)at(0,1){$\mod\L$};
\node(P2)at(-2,0){$\Fac \L e_{2}$};
\node(P1)at(2,0){$\Fac \L e_{1}$};
\node(L)at(0,-1){$\Fl\L$};
\node(M2P2)at(-2,-2){$\Fac \L e_{2} \cap \Fl\L$};
\node(P1M1)at(2,-2){$\Fac \L e_{1} \cap \Fl\L$};
\node(M2)at(-2,-3){$\Fac S_{2}$};
\node(M1)at(2,-3){$\Fac S_{1}$};
\node(0)at(0,-4){$0$};
\draw[thick, ->] (mod)--(P2);
\draw[thick, ->] (mod)--(P1);
\draw[thick, ->] (mod)--(L);
\draw[thick, ->] (P2)--(M2P2);
\draw[thick, ->] (P1)--(P1M1);
\draw[thick, ->] (L)--(M2P2);
\draw[thick, ->] (L)--(P1M1);
\draw[thick, ->] (M2P2)--(M2);
\draw[thick, ->] (P1M1)--(M1);
\draw[thick, ->] (M2)--(0);
\draw[thick, ->] (M1)--(0);
\end{tikzpicture}.
\end{equation*}
\end{example}
\subsection{Bass orders of type (V) and its Auslander orders}\label{subsection-Bass-V}
We say that an $R$-order is of \emph{CM-finite type} if the number of isomorphic classes of indecomposable $\L$-lattices is finite.

Let $n$ be a non-negative integer and $\L$ be the following $R$-order
\[
\L = \left[\begin{array}{cc}
R & R \\
\mfm^{n} & R
\end{array}
\right].
\]
This $\L$ is known as a Bass order of type $(\rm V)$, and was studied by \cite{DK, Hijikata-Nishida}.
It has two indecomposable projective modules
\[
P_{1}=\begin{bmatrix} R \\ \mfm^{n} \end{bmatrix}, \qquad
P_{2}=\begin{bmatrix} R \\ R \end{bmatrix}.
\]
We denote by $M_{1}$ a cokernel of the map $x^n : P_{2} \to P_{1}$ as follows, and denote by $M_{2}$ a cokernel of the natural inclusion map $P_{1}\to P_{2}$ as follows
\[
M_{2}=\begin{bmatrix} 0 \\ R/\mfm^{n} \end{bmatrix}, \qquad
M_{1}=\begin{bmatrix} R/\mfm^n \\ 0 \end{bmatrix}.
\]
Then we have four non-trivial silting $\L$-modules $M_{2}\oplus P_{2}$, $P_{1}\oplus M_{1}$, $M_{2}$ and $M_{1}$.
By a direct calculation, the exchange quiver $Q(\silt\,\L)$ is as follows
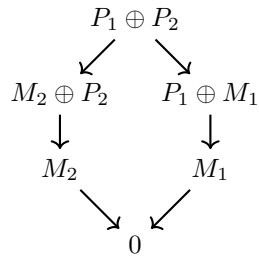
\begin{equation*}\label{example-triangular-mutation-quiver}
\begin{tikzpicture}
\node(L)at(0,0){$P_{1}\oplus P_{2}$};
\node(M2P2)at(-1,-1){$M_{2}\oplus P_{2}$};
\node(P1M1)at(1,-1){$P_{1}\oplus M_{1}$};
\node(M2)at(-1,-2){$M_{2}$};
\node(M1)at(1,-2){$M_{1}$};
\node(0)at(0,-3){$0$};
\draw[thick, ->] (L)--(M2P2);
\draw[thick, ->] (L)--(P1M1);
\draw[thick, ->] (M2P2)--(M2);
\draw[thick, ->] (P1M1)--(M1);
\draw[thick, ->] (M2)--(0);
\draw[thick, ->] (M1)--(0);
\end{tikzpicture}.
\end{equation*}

The $R$-order $\L$ is of CM-finite type.
The Auslander-Reiten quiver of $\CM\L$ is the following
\begin{equation}\label{AR-quiver-V}
\begin{tikzpicture}
\node(1)at(0,0){$\begin{bmatrix} R \\ R \end{bmatrix}$};
\node(2)at(2,0){$\begin{bmatrix} R \\ \mfm \end{bmatrix}$};
\node(cdots)at(4,0){$\cdots$};
\node(n)at(6,0){$\begin{bmatrix} R \\ \mfm^{n-1} \end{bmatrix}$};
\node(n+1)at(8.3,0){$\begin{bmatrix} R \\ \mfm^{n} \end{bmatrix},$};
\path[->, thick, font=\scriptsize ,>=angle 45]
([yshift= 2pt]1.east) edge node[above] {$\alpha_{1}$} ([yshift= 2pt]2.west)
([yshift= -2pt]2.west) edge node[below] {$\beta_{1}$} ([yshift= -2pt]1.east);
\path[->, thick, font=\scriptsize ,>=angle 45]
([yshift= 2pt]2.east) edge node[above] {$\alpha_{2}$} ([yshift= 2pt]cdots.west)
([yshift= -2pt]cdots.west) edge node[below] {$\beta_{2}$} ([yshift= -2pt]2.east);
\path[->, thick, font=\scriptsize ,>=angle 45]
([yshift= 2pt]cdots.east) edge node[above] {$\alpha_{n-1}$} ([yshift= 2pt]n.west)
([yshift= -2pt]n.west) edge node[below] {$\beta_{n-1}$} ([yshift= -2pt]cdots.east);
\path[->, thick, font=\scriptsize ,>=angle 45]
([yshift= 2pt]n.east) edge node[above] {$\alpha_{n}$} ([yshift= 2pt]n+1.west)
([yshift= -2pt]n+1.west) edge node[below] {$\beta_{n}$} ([yshift= -2pt]n.east);
\draw (2) edge [->, thick, dashed, out=110, in=70, looseness=6] (2);
\draw (n) edge [->, thick, dashed, out=110, in=70, looseness=6] (n);
\end{tikzpicture}
\end{equation}
where dotted arrows are the Auslander-Reiten translation on $\CM\L$.
Each $\beta$ is the natural inclusion map and each $\alpha$ is the multiplication by $x$.

Let $\Gamma$ be the Auslander order of $\CM\L$, that is, the endomorphism algebra of a basic additive generator $M=\bigoplus_{i=0}^{n}\begin{bmatrix} R \\ \mfm^{i} \end{bmatrix}$ of $\CM\L$.
We denote by $Q$ the Auslander-Reiten quiver (\ref{AR-quiver-V}).
Then we have
\begin{align*}
\Gamma/\mfm\Gamma \simeq \widehat{kQ}/ I,
\end{align*}
where $I$ is the two-sided ideal of $\widehat{kQ}$ generated by all $2$-cycles of $Q$.

We denote by $\mathfrak{S}_{n+2}$ the Symmetric group of degree $n+2$.
We regard $\mathfrak{S}_{n+2}$ as a poset by the right weak order, see \cite{DIRRT} for details.
\begin{thm}\label{thm-msilt-Bass-V-Aus}
Let $\Gamma$ be as above.
Then we have the following isomorphism of posets.
\[
\silt\Gamma \simeq \mathfrak{S}_{n+2}.
\]
\end{thm}
\begin{proof}
Let $\Pi$ be the preprojective algebra of type $A_{n+1}$ and $I_{2}\subset\Pi$ the two-sided ideal of $\Pi$ generated by all $2$-cycles.
We have isomorphisms  $\Gamma/\mfm\Gamma \simeq \widehat{kQ}/ I \simeq \Pi/I_{2}$, where $I\subset \widehat{kQ}$ is the two-sided ideal generated by all $2$-cycles.
We have isomorphisms of posets
\[
\silt\Gamma \simeq \silt(\Gamma/\mfm\Gamma) \simeq \silt(\Pi/I_{2}) \simeq \silt\Pi \simeq \mathfrak{S}_{n+2},
\]
where the first isomorphism comes from Theorem \ref{thm-reduction}(b), the third one comes from \cite[Proposition 5.7]{DIRRT} and the fourth one comes from \cite{Mizuno}.
\end{proof}
\section*{Acknowledgements}
The author would like to thank Professor Osamu Iyama for many supports and helpful comments.


\begin{thebibliography}{99}

\bibitem[Ad]{Adachi}
T. Adachi,
\emph{The classification of {$\tau$}-tilting modules over {N}akayama algebras},
J. Algebra 452 (2016), 227--262.

\bibitem[AIR]{Adachi-Iyama-Reiten}
T. Adachi, O. Iyama, I. Reiten,
\emph{$\tau$-tilting theory},
Compos. Math. 150 (2014), no. 3, 415--452.

\bibitem[Ai]{Aihara}
T. Aihara,
\emph{Tilting-connected symmetric algebras},
Algebr. Represent. Theory 16 (2013), no. 3, 873--894. 

\bibitem[AI]{Aihara-Iyama}
T. Aihara, O. Iyama,
\emph{Silting mutation in triangulated categories},
J. Lond. Math. Soc. (2) 85 (2012), no. 3, 633--668.


\bibitem[AHK]{AHK}
L. Angeleri H\"{u}gel, D. Happel, H. Krause,
\emph{Handbook of tilting theory},
London Mathematical Society Lecture Note Series, 332. Cambridge University Press, Cambridge, 2007. viii+472 pp.

\bibitem[AMV]{AngeleriHugel-Marks-Vitoria}
L. Angeleri H\"{u}gel, F. Marks, J. Vit\'oria,
\emph{Silting modules},
Int. Math. Res. Not. IMRN 4 (2016), 1251--1284

\bibitem[ASS]{ASS}
I. Assem, D. Simson, A. Skowro$\acute{{\rm n}}$ski, 
\emph{Elements of the representation theory of associative algebras. Vol. 1. Techniques of representation theory}, 
London Mathematical Society Student Texts, 65. Cambridge University Press, Cambridge, 2006.

\bibitem[BB]{Brenner-Butler}
S. Brenner, M. C. R. Butler, 
{\it Generalizations of the Bernstein-Gelfand-Ponomarev reflection functors}, 
Lecture Notes in Math. 832, Springer, Berlin-New York, 1980.

\bibitem[BIRS]{BIRSc} A. Buan, O. Iyama, I. Reiten, J. Scott, \emph{Cluster structures for $2$-Calabi-Yau categories and unipotent groups}, Compos. Math. 145 (2009), no. 4, 1035--1079.


\bibitem[BMRRT]{BMRRT}
A. Buan, R. Marsh, M. Reineke, I. Reiten, G. Todorov,
\emph{Tilting theory and cluster combinatorics},
Adv. Math. 204 (2006), 572--618.

\bibitem[CR]{Curtis-Reiner}
C. W. Charles, I. Reiner,
\emph{Methods of representation theory. Vol. I. With applications to finite groups and orders},
Pure and Applied Mathematics,
John Wiley \& Sons,
Inc., New York, 1981.

\bibitem[DIJ]{DIJ}
L. Demonet, O. Iyama, G. Jasso,
\emph{$\tau$-tilting finite algebras, bricks and $g$-vectors},
Int. Math. Res. Not. IMRN 2019, no. 3, 852--892.

\bibitem[DIRRT]{DIRRT}
L. Demonet, O. Iyama, N. Reading, I. Reiten, H. Thomas,
\emph{Lattice theory of torsion classes},
arXiv:1711.01785v2.

\bibitem[DK]{DK}
Yu. A. Drozd, V. V. Kirichenko,
\emph{On quasi-Bass orders},
Izv. Akad. Nauk SSSR Ser. Mat. 36 (1972) 328--370.

\bibitem[E]{Eisele}
F. Eisele, \emph{Bijections of silting complexes and derived Picard groups}, arXiv:2101.06258.

\bibitem[EJR]{EJR}
F. Eisele, G. Janssens, T. Raedschelders,
\emph{A reduction theorem for $\tau$-rigid modules},
Math. Z. 290 (2018), no. 3-4, 1377--1413.

\bibitem[G]{Gabriel}
P. Gabriel,
\emph{Des cat\'egories ab\'eliennes},
Bull. Soc. Math. France 90 (1962) 323--448.

\bibitem[GLS06]{GLS06}
C. Geiss, B. Leclerc, J. Schr\"{o}er, \emph{Rigid modules over preprojective algebras}, Invent. Math. 165 (2006), 589--632.

\bibitem[GLS13]{GLS13}
C. Geiss, B. Leclerc, J. Schr\"{o}er, \emph{Cluster algebras in algebraic Lie theory}, Transform. Groups 18 (2013), no. 1, 149--178. 


\bibitem[Gn]{Gnedin} W. Gnedin, \emph{Silting theory of orders modulo a regular sequence}, Representation Theory of Quivers and Finite Dimensional Algebras, Oberwolfach Rep. 17, No. 1, 182--185 (2020).


\bibitem[H]{Happel}
D. Happel, 
\emph{Triangulated categories in the representation theory of finite-dimensional algebras},
London Mathematical Society Lecture Note Series, 119. Cambridge University Press, Cambridge, 1988.


\bibitem[HN]{Hijikata-Nishida}
H. Hijikata, K. Nishida,
\emph{Bass orders in nonsemisimple algebras},
J. Math. Kyoto Univ. 34 (1994), no. 4, 797--837.

\bibitem[IT]{Ingalls-Thomas} C. Ingalls, H. Thomas, \emph{Noncrossing partitions and representations of quivers}, Compos. Math. 145 (2009), no. 6, 1533--1562.

\bibitem[IJY]{Iyama-Jorgensen-Yang}
O. Iyama, P. Jorgensen, D. Yang,
\emph{Intermediate co-$t$-structures, two-term silting objects, $\tau$-tilting modules, and torsion classes},
Algebra Number Theory 8 (2014), no. 10, 2413--2431.

\bibitem[IK]{Iyama-Kimura}
O. Iyama, Y. Kimura,
\emph{Classifying subcategories of modules over Noetherian algebras},
arXiv:2106.00469.

\bibitem[IR]{Iyama-Reiten}
O. Iyama, I. Reiten,
\emph{Fomin-Zelevinsky mutation and tilting modules over Calabi-Yau algebras},
Amer. J. Math. 130 (2008), no. 4, 1087--1149.

\bibitem[IW]{Iyama-Wemyss}
O. Iyama, M. Wemyss,
\emph{Maximal modifications and Auslander-Reiten duality for non-isolated singularities}, Invent. Math. 197 (2014), no. 3, 521--586.

\bibitem[IYa]{Iyama-Yang}
O. Iyama, D. Yang,
\emph{Silting reduction and Calabi--Yau reduction of triangulated categories},
Trans. Amer. Math. Soc. 370 (2018), no. 11, 7861--7898.

\bibitem[IYo]{Iyama-Yoshino}
O. Iyama, Y. Yoshino,
\emph{Mutation in triangulated categories and rigid Cohen-Macaulay modules},
Invent. Math. 172 (2008), no. 1, 117--168.

\bibitem[J]{Jasso}
G. Jasso,
\emph{Reduction of $\tau$-tilting modules and torsion pairs},
Int. Math. Res. Not. IMRN 2015, no. 16, 7190--7237.

\bibitem[KY]{KY}
B. Keller, D. Yang,
\emph{Derived equivalences from mutations of quivers with potential}, Adv. Math. 226 (2011), no. 3, 2118--2168.

\bibitem[KV]{KV}
B. Keller, D. Vossieck,
\emph{Aisles in derived categories},
Deuxi\`{e}me Contact Franco-Belge en Alg\`{e}bre (Faulx-les-Tombes, 1987), Bull. Soc. Math. Belg. Ser. A40 (1988) 239--253.

\bibitem[KM]{Kimura-Mizuno}
Y. Kimura, Y. Mizuno,
\emph{Two-term tilting complexes for preprojective algebras of non-Dynkin type}, Comm. Algebra 50 (2022), no. 2, 556--570.

\bibitem[M]{Mizuno}
Y. Mizuno,
\emph{Classifying $\tau$-tilting modules over preprojective algebras of Dynkin type},
Math. Z. 277 (2014), no. 3-4, 665--690.

\bibitem[R]{Rickard}
J. Rickard,
\emph{Morita theory for derived categories},
J. London Math. Soc. (2) 39 (1989), no. 3, 436--456.

\bibitem[RS]{RS}
C. Riedtmann, A. Schofield,
\emph{On a simplicial complex associated with tilting modules},
Comment. Math. Helv. 66(1) (1991) 7--78.

\bibitem[S]{Smalo}
S. O. Smal\o,
\emph{Torsion theory and tilting modules},
Bull. Lond. Math. Soc. 16 (1984), 518--522.

\bibitem[SW]{Stanley-Wang}
D. Stanley, B. Wang,
\emph{Classifying subcategories of finitely generated modules over a Noetherian ring},
J. Pure Appl. Algebra 215 (2011), no. 11, 2684--2693.

\end{thebibliography}
\end{document}